\documentclass[reqno]{amsart}

\usepackage{geometry}
\usepackage{amsmath,amssymb,amsthm,amsfonts}
\usepackage{hyperref}
\usepackage{mathrsfs}
\usepackage{appendix}
\usepackage{graphicx}
\usepackage{setspace}

\allowdisplaybreaks[4]

\usepackage{color}

\newtheorem{lemma}{Lemma}[section]
\newtheorem{theorem}{Theorem}[section]
\newtheorem{definition}{Definition}[section]
\newtheorem{proposition}{Proposition}[section]
\newtheorem{remark}{Remark}[section]

\numberwithin{equation}{section}
\newcommand\bu{\mathbf{u}}

\newcommand{\beq}{\begin{equation}}
\newcommand{\eeq}{\end{equation}}
\newcommand{\ben}{\begin{eqnarray}}
\newcommand{\een}{\end{eqnarray}}
\newcommand{\beno}{\begin{eqnarray*}}
\newcommand{\eeno}{\end{eqnarray*}}

\begin{document}
\title[ Global strong solutions for 1D compressible NSCH equations with vacuum]{Global strong solutions for 1D compressible Navier-Stokes/Cahn-Hilliard equations with vacuum}
\thanks{$^*$Corresponding author}
\thanks{{\it Keywords}: Compressible Navier--Stokes; Cahn--Hilliard system;
initial vacuum; strong solutions.}
\thanks{{\it AMS Subject Classification}: 35A01, 35A02, 35Q35, 76N10}%
\author[Shijin Ding]{Shijin Ding}
\address[S. Ding]{School of Mathematical Sciences, South China Normal University,
Guangzhou, 510631, China}
\email{dingsj@scnu.edu.cn}
\author[Yinghua Li]{Yinghua Li}
\address[Y. Li]{School of Mathematical Sciences, South China Normal University,
Guangzhou, 510631, China}
\email{yinghua@scnu.edu.cn}
\author[Yuanxiang Yan]{Yuanxiang Yan*}
\address[Y. Yan]{School of Mathematics and Systems Science, Guangdong Polytechnic Normal University,
Guangzhou, 510655, China}
\email{yuanxiangyan@m.scnu.edu.cn}
\author[Haoran Zheng]{Haoran Zheng}
\address[H. Zheng]{School of Mathematical Sciences, South China Normal University,
Guangzhou, 510631, China}
\email{hrzheng@scnu.edu.cn}

\date{\today}

\begin{abstract}
In this paper, we study the initial-boundary value problem of the 1D compressible Navier--Stokes/Cahn--Hilliard system with vacuum. We establish the global existence and uniqueness of strong solutions to this initial-boundary value problem. No any initial compatibility conditions are required via time weighted techniques, which leads to  a loss of regularity near the initial time. Therefore, the uniqueness of solutions obtained in this paper is even more challenging. To address this issue, we establish refined growth estimates and singular-in-time weighted energy estimates that induce a Gronwall-type structure, which ultimately allows us to close the uniqueness proof in Eulerian coordinates without passing to Lagrangian coordinates.
\end{abstract}

\maketitle


\section{Introduction}

In this paper, we investigate the compressible Navier--Stokes/Cahn--Hilliard (NSCH) system, which describes a diffusive interface model for the two-phase flow of viscous fluids. The model considered here was first deduced by Lowengrub and Truskinovsky \cite{LT98}. It has been modified and studied by Abels and Feireisl \cite{AF08} in the following form
\begin{equation}\label{NSCH_1}
	\left\{
	\begin{aligned}
		& \partial_t \rho + \operatorname{div}(\rho \bu)=0,  \\
		& \rho \partial_t \bu + \rho \bu \cdot \nabla\bu - \operatorname{div}\mathbb{S} + \nabla p = -\operatorname{div}\left(\nabla \phi \otimes \nabla \phi - \frac{1}{2}|\nabla \phi|^2 \mathbb{I}\right),\\
		& \rho \partial_t \phi + \rho \bu \cdot \nabla \phi = \Delta \mu,\\
		& \rho \mu =  - \Delta \phi + \rho \frac{\partial f}{\partial \phi}.
	\end{aligned}
	\right.
\end{equation}
Here, $\rho$, $\bu$, $\phi$, and $\mu$ denote the total density, the mean velocity of the fluid mixture, the phase-field variable, and the chemical potential, respectively. The viscous stress tensor $\mathbb{S}$ satisfies
$$
\mathbb{S} =\lambda(\phi)\left(\nabla \mathbf{u}+\nabla^{\top} \mathbf{u}-\frac{2}{3} \operatorname{div}  \mathbf{u} \mathbb{I}\right)+\eta(\phi) \operatorname{div} \mathbf{u} \mathbb{I} .
$$
The functions $\lambda(\phi)>0$ and $\eta(\phi)\geq 0$ are the shear and bulk viscosities, respectively.
The free energy density $f$ takes the form
$$
f(\rho,\phi)=\frac{\rho^{\gamma-1}}{\gamma-1}+\Phi(\phi),
\qquad
\Phi(\phi)=\frac{1}{4}(\phi^{2} - 1)^{2},
$$
and is related to the pressure through the isentropic equation of state
$p=\rho^{2} \frac{\partial f}{\partial \rho} = \rho^{\gamma}$. Here $\Phi(\phi)$ is also known as Landau potential \cite{LL13}.

In the special case of single phase fluid, i.e., $\phi =1$ or $\phi=-1$. The system \eqref{NSCH_1} reduces to Navier--Stokes equations. Let us recall about the results about Navier--Stokes equations with initial density vacuum.

In the presence of vacuum, i.e., when the initial density vanishes on some region, there  has been  a number of works on  the compressible Navier–Stokes equations since Lions \cite{L98} established the global existence of weak solutions for the isentropic case with the pressure law $p=a\rho^{\gamma}$ where $a>0$ and the adiabatic exponent satisfies $\gamma \ge \frac{9}{5}$. This result was subsequently extended by Feireisl, Novotn\'y, and Petzeltov\'a \cite{FNP01} to the range $\gamma \ge \frac{3}{2}$, and further by Jiang and Zhang \cite{JS01, JS03} to the case $\gamma>1$ under the additional assumptions of spherical symmetry or axisymmetry. For more recent developments, we refer to Bresch and Jabin \cite{BJ18}, where more general stress tensors and pressure laws are allowed. In the context of the full compressible Navier–Stokes equations (including energy equation), Feireisl \cite{F04} proved global existence of weak solutions under certain structural assumptions on the viscous and heat-conductive coefficients as well as the equations of state. Nevertheless, the uniqueness of such weak solutions remains an open problem.

For the uniqueness in the case involving vacuum, Salvi and Straškraba \cite{SS93} first proved the local existence and uniqueness of strong solutions to the isentropic compressible Navier–Stokes equations with suitably regular initial data and the following compatibility condition
\begin{align}\label{CB}
-\mu \Delta u_{0}  - (\mu + \lambda) \nabla {\rm div} u_{0} - \nabla p(\rho_{0}) = \sqrt{\rho_{0}}g,
\end{align}
for some $ g \in  L^{2}$. Since the work \cite{SS93}, the compatibility condition and its necessary modifications are widely used, as the standard assumptions, in many paper concerning the studies of the existence and uniqueness of strong solutions with initial vacuum allowed, see \cite{CK03, CCK04, CK06}. Later, Huang, Li and Xin \cite{HLX12} established a global-in-time existence result under a small initial energy assumption, see \cite{GHW13, HL18, HW15, LX19, Z15} for some further developments in this direction. Different from the multi-dimensional case, in the one-dimensional case, the global well posedness of strong solutions can be established for arbitrary large initial data, see \cite{JLY14, L19, L20}.

In contrast to the well-developed theory for Navier--Stokes equations, available results for compressible NSCH systems are rather limited. A first compressible diffuse--interface model was proposed by Lowengrub and Truskinovsky \cite{LT98}; see also \cite{EBS25} for recent extensions to N-phase mixtures. A simplified variant was later introduced by Abels and Feireisl \cite{AF08}, which corresponds precisely to system \eqref{NSCH_1} considered in the present work.
More recently, Elbar and Poulain \cite{EP24} established the global-in-time existence of weak solutions to \eqref{NSCH_1} with an additional friction term under the condition $\gamma>6$, which was subsequently relaxed to $\gamma>\tfrac{3}{2}$ by Basari\'c and Giorgini \cite{BG25} for a Flory--Huggins (Boltzmann--Gibbs) type free energy. On the other hand, Kotschote and Zacher \cite{KZ15} proved the local-in-time existence and uniqueness of strong solutions to the model proposed in \cite{LT98}. Further developments include the existence of global weak solutions with dynamic boundary conditions \cite{CFMMPP19}, the low Mach number limit \cite{ALN24,LX25}, the analysis of stationary problems in \cite{LW20,LW22}.  Ding and Li \cite{D-L} established global classical solution in the one-dimensional case without vacuum.

For isentropic 1D compressible Navier-Stokes/Allen-Cahn (NSAC) system, Ding, Li and Luo \cite{DLL13} obtained  global classical solution without vacuum. Later, this result was extended by Chen and Guo \cite{C-G} to the vacuum case, under the following compatibility conditions
\begin{align}
\mu u_{0xx} - p(\rho_{0})_{x} - \frac{1}{2}(\phi_{0x}^{2})_{x} = \rho_{0} g_{1}, \label{C11}
\\
\mu_{0} = \rho_{0} g_{2}, \label{C22}
\end{align}
for $(g_{1}, g_{2xx}) \in H^{1} \times L^{2}$. Moreover, under the assumption of  compatibility conditions similar to  \eqref{C11} and \eqref{C22}, Chen and Zhu \cite{C-Z} further extended this result to the case where the viscosity is given  $\mu(\rho) = 1 + \rho^{\alpha}$ with $ 2 \le \alpha \le \gamma$. Meanwhile, Su \cite{S21} established the global existence and uniqueness of strong and classical solutions to 1D compressible NSAC system with density-dependent viscosity and obtained the large time behavior of the velocity.
In the three-dimensional setting, Li, Zheng, and Zhou \cite{LZZ26} established the local existence and uniqueness of strong solutions in the presence of vacuum, under the following compatibility condition on the initial phase-field variable:
\begin{align}
		\Delta \phi_0 & = \rho_0 h, \quad \text{in } \Omega, \\
		\partial_{n} \phi_0 & = 0, \quad \quad
		 \text{on } \partial\Omega,
\end{align}
for some $h\in L^2(\Omega)$.

It is worth pointing out that the compatibility condition \eqref{CB} or its natural amendments play a crucial role in the well‑posedness theories. Consequently, they have been accepted as standard assumptions for establishing the well‑posedness of the compressible Navier–Stokes equations and related models in the presence of vacuum. One should note that these conditions impose restrictive constraints on the initial data, both in the vacuum region and in a neighborhood of the vacuum--nonvacuum interface. Indeed, by the compatibility condition \eqref{CB}, the initial velocity $u_{0}$ is destined to obey
$$
-\mu \Delta u_{0}  - (\mu + \lambda) \nabla {\rm div} u_{0} = 0,
$$
in the vacuum region, which however seem not physically relevant.

Based on the above analysis, an alternative well‑posedness theory that avoids any initial compatibility conditions such as \eqref{CB} is  both mathematically and physically importance. The first contribution in this direction was made by Li \cite{L17} for the inhomogeneous incompressible Navier–Stokes equations, where local well‑posedness was successfully established without imposing any compatibility conditions on the initial data. Further progress aimed at reducing the regularity requirements on the initial density can be found in Danchin and Mucha \cite{DM19}. Subsequently, an analogous local well‑posedness theory free of initial compatibility conditions was independently developed for the isentropic compressible Navier–Stokes equations by Gong et al. \cite{GLLZ20} and Huang \cite{H21}. For the full compressible Navier–Stokes equations, the local well‑posedness theory was achieved by Lai, Xu and Zhang \cite{LXZ22},  where they removed the compatibility  condition of $u_{0}$ but still required compatibility condition of $\theta_{0}$.  Later, Li and Zheng \cite{LZ23} extended this result to the case without any compatibility conditions on the initial data. For the inhomogeneous incompressible NSAC system, the global existence of unique strong solutions to the 3D Cauchy problem and the initial boundary value problem is established by Li, Xie and Yan \cite{LXY25}. Here, the initial vacuum is allowed and no compatibility conditions are required for the initial data. However, for the compressible NSCH equations, to the best of our knowledge, the desired well-posedness theory without any compatibility conditions on the initial data has not been established.

The aim of this paper is to establish the desired global well‑posedness theory for 1D compressible NSCH system without any extra compatibility conditions beyond the essential smoothness conditions on the initial data. In one spatial dimension, the viscous stress reduces to
$\mathbb{S} = \nu_{\mathrm{eff}}(\phi) u_x$,
where $\nu_{\mathrm{eff}}(\phi)=\frac{4}{3}\lambda(\phi)+\eta(\phi)>0$.
For simplicity, we assume $\nu_{\mathrm{eff}}(\phi)\equiv1$.  Then the compressible NSCH system \eqref{NSCH_1} in one dimension reduces to the following form
\begin{align}\label{nsch-1d}
	\begin{cases}
		\rho_{t} + (\rho u)_{x} = 0, \\[2pt]
		\rho u_{t} + \rho u u_{x} + (\rho^{\gamma})_{x}
		= u_{xx} - \frac{1}{2}(\phi_{x}^{2})_{x}, \\[2pt]
		\rho \phi_{t} + \rho u \phi_{x} = \mu_{xx}, \\[2pt]
		\rho \mu = -\phi_{xx} + \rho(\phi^{3}-\phi),
	\end{cases}
\end{align}
for $(x,t)\in(0,1)\times(0,T)$.

The system \eqref{nsch-1d} is supplemented with the boundary conditions
\begin{align}\label{bc}
	(u,\phi_x,\mu_x)\big|_{x=0,1}=(0,0,0),
\end{align}
and the initial conditions
\begin{align}\label{ic}
	(\rho,\rho u,\rho\phi)\big|_{t=0}=(\rho_0,\rho_0u_0,\rho_0\phi_0).
\end{align}
We emphasize that the initial density $\rho_0 \ge 0$ is allowed to vanish on an open subset of $I$; in other words, the presence of an initial vacuum is permitted.

Before stating our main results, we first introduce some necessary notations used throughout this paper and state the definition of solutions to be established. For $1\le p \le \infty$, we denote
\begin{align*}
\begin{cases}
I=(0,1), ~Q_{T} = I \times [0,T] ~ {\rm for}~ T>0,
\\
L^{p} = L^{p}(I), ~ W^{k, p} = W^{k, p}(I), ~ H^{k} = W^{k, 2},
\\
\| (f_{1}, f_{2}, \cdots, f_{N}) \|_{X} = \sum \limits_{i=1}^{N}  \| f_{i} \|_{X}.
\end{cases}
\end{align*}
Without loss of generality, we assume throughout this paper that
\begin{align*}
\int_{I} \rho_{0} (x) {\rm d}x = 1.
\end{align*}

The strong solutions to be established in this paper are defined as follows.
\begin{definition}\label{Define}
Given positive time $T\in (0,\infty)$. Assume that the initial data satisfy
\begin{align*}
0 \le \rho_{0} \in H^{1}, \qquad u_{0} \in H_{0}^{1},\qquad \phi_{0} \in H^{1},
\end{align*}
and that $\rho_{0}$ is not identically zero. A triple $(\rho, u, \phi)$ is called a strong solution  to the problem \eqref{nsch-1d}--\eqref{ic} in $I\times (0,T)$, if it has the regularities
\begin{alignat}{2}
& 0 \le \rho \in L^{\infty}(0,T;H^{1}), &\qquad& \rho_{t} \in L^{\infty}(0,T;L^{2}),
 \nonumber \\
& \rho u  \in C([0,T];L^{2}),  &\qquad& u \in L^{\infty}(0,T;H_{0}^{1}) \cap L^{2}(0,T; H^{2}),
\nonumber\\
& \sqrt{\rho} u_{t} \in L^{2}(0,T;L^{2}),  &\qquad& t u \in L^{\infty}(0,T;H^{2}) \cap L^{2}(0,T; H^{3}),
\nonumber\\
& \sqrt{t} u_{t} \in  L^{\infty}(0,T; H^{1}), &\qquad& \phi \in L^{\infty}(0,T; H^{1}) \cap L^{2}(0,T; H^{2}),
\nonumber\\
& \rho \phi  \in C([0,T];L^{2}), &\qquad& \sqrt{t} \phi_{xx} \in L^{\infty}(0,T;H^{1}) ,
\nonumber\\
& \sqrt{t} \phi_{t} \in L^{\infty}(0,T;H^{1}),  &\qquad& t\phi_{xxt} \in L^{2}(0,T; L^{2}),
\nonumber\\
&\sqrt{t} \mu \in  L^{\infty}(0,T; H^{1}),  &\qquad& \mu \in  L^{2}(0,T; H^{1}),
\nonumber\\
& t\mu_{xx} \in L^{\infty}(0,T;L^{2}), &\qquad&  t \mu_{t} \in  L^{2}(0,T; H^{1}) .
\nonumber
\end{alignat}
satisfies equations \eqref{nsch-1d} a.e. in $I \times (0, T)$, and fulfills the initial condition \eqref{ic}.
\end{definition}

\begin{remark}
Thanks to  $\rho u,  \rho \phi\in C([0,T];L^{2})$,  the initial values of $\rho u$ and $\rho \phi$ are well-defined.
\end{remark}

The main result of this paper can be summarized as follows.
\begin{theorem}\label{TH}
Assume that the initial data satisfy
\begin{align*}
0 \le \rho_{0} \in H^{1}, \qquad u_{0} \in H_{0}^{1},\qquad \phi_{0} \in H^{1},
\end{align*}
and that $\rho_{0}$ is not identically zero. Then for any $T>0$, there exists a unique global strong solution $(\rho, u, \phi)$ to the problem \eqref{nsch-1d}--\eqref{ic} in $I\times (0,T)$.
\end{theorem}

\begin{remark}
	A notable feature of the present work is that no initial compatibility conditions
	are imposed on the data. By means of singular-in-time weighted energy estimates,
	we establish the global existence and uniqueness of strong solutions with vacuum.
	This approach is partly inspired by the time-weighted energy method in~\cite{LZ23}.
\end{remark}

\begin{remark}
These assumptions on the initial value $\phi_{0}$ are weaker than those in \cite{C-G, C-Z, S21}. In fact, $\phi_{0} \in H^{1}_{0}\cap H^{2}$ was required in \cite{C-G, C-Z, S21}, where the phase field variable $\phi$ is subject to Dirichlet boundary condition.  In Theorem \ref{TH}, we replace it with the Neumann boundary condition, which is more physically meaningful.
\end{remark}


A fundamental difficulty remains largely open in the analysis of compressible NSCH systems, namely the presence of vacuum. From a mathematical point of view, vacuum regions cause degeneracy in the momentum equation, leading to a loss of coercivity and preventing the direct application of classical energy methods. Moreover, the strong coupling between the hydrodynamic variables and the phase-field equation introduces additional nonlinear effects, which significantly complicate the derivation of higher-order a priori estimates. These challenges are particularly severe when seeking strong solutions that allow initial vacuum, even in one spatial dimension.

Let us state the key ideas in our arguments.  It is noted that in the previous literature \cite{C-G, C-Z}, the compressible Navier-Stokes/Allen-Cahn system with initial density vacuum was considered, where the phase field variable $\phi$ is subject to Dirichlet boundary condition.  Such a boundary condition is adopted in order to overcome technical difficulties by applying the Poincar\'e inequality. Nevertheless, the Neumann boundary condition for $\phi$ is more physically meaningful. To address the challenges posed by the Neumann boundary condition, we establish the following inequality
\begin{align*}
\| \partial_{t}^{i} \phi \|_{L^{\infty}} \le C \| \sqrt{\rho} \partial_{t}^{i}\phi \|_{L^{2}}+ C \| \partial_{x}\partial_{t}^{i}\phi  \|_{L^{2}}, \quad i=0,1.
\end{align*}

Due to  the difficulties caused by the low regularities and lack of compatibility conditions on the initial data, leading to weaker regularities of the solutions, the uniqueness of solutions obtained in the current paper  is even more challenging. Our strategies of proving the uniqueness are illustrated as follows. Let $(\rho_{1}, u_{1}, \phi_{1})$ and $(\rho_{2}, u_{2}, \phi_{2})$ be two solutions
with the same initial data and denote by $(\rho, u, \phi)$ their subtraction. Then, we have singular-in-time weighted energy estimates of the form
\begin{align}\label{I-w}
&   \mathcal{A}(t) + \int_{0}^{t} \left( \frac{\| \rho \|_{L^{2}}^{2}}{t^{2}} +  \| u_{x}\|_{L^{2}}^{2} + \| \mu_{x}\|_{L^{2}}^{2} \right) (s) {\rm d}s
 \le C \int_{0}^{t}  \mathcal{B}(s)  \mathcal{A}(s) {\rm d}s,
\end{align}
where $ \mathcal{B}(s) \in L^{1}(0,t) $ and
\begin{align*}
\mathcal{A}(t) := \frac{\| \rho \|_{L^{2}}^{2}}{t} +   \| \sqrt{\rho_{1}} u \|_{L^{2}}^{2} + \| \sqrt{\rho_{1}} \phi\|_{L^{2}}^{2} + \| \phi_{x}\|_{L^{2}}^{2}.
\end{align*}
Note that \eqref{I-w} meets the Gronwall type structure. It remains to guarantee that the quantity with singular weights $\mathcal{A}(t)$ tends to zero when approaching the initial time.  Thus, one needs to show that the initial values of $\sqrt{\rho_{1}} u$ and  $\sqrt{\rho_{1}} \phi$ are identically zero. However, it is a subtle issue to verify this in the Euler coordinates, as the initial condition is $(\rho_{1} u_{1}, \rho_{1} \phi_{1}) \big|_{t=0} = (\rho_{2} u_{2}, \rho_{2} \phi_{2}) \big|_{t=0}$. Because of the technical challenges noted above, Li-Zheng \cite{LZ23} established the uniqueness initially in Lagrangian coordinates and later convert it back to Eulerian coordinates. Comparing with \cite{LZ23}, we establish the uniqueness in Eulerian coordinates rather than in Lagrangian coordinates. Besides, we need to additionally show that the higher-order term
\begin{align}\label{L-key-e}
\lim \limits_{t \to 0}\| \phi_{x}\|_{L^{2}}^{2}(t) = 0.
\end{align}
Firstly, we can find that the continuity $\phi_{x} \in C([0,T];L^{2})$ is not available, since the low regularities of solutions,  lack of compatibility conditions on the initial data, and the initial data $\rho \phi \big|_{t=0} = \rho_{0} \phi_{0}$. To establish \eqref{L-key-e}, we consider the growth estimate of $ \| \phi_{x}\|_{L^{2}}^{2}(t) $  and have the following observation
\begin{align}
\| \phi_{x}\|_{L^{2}}^{2}(t)  \le \| \sqrt{\rho_{1}} \phi \|_{L^{2}}(t) \| \sqrt{\rho_{1}} \mu \|_{L^{2}}(t) + \mathcal{O}(t^{\frac{1}{2}}),
\end{align}
Consequently, we proceed to establish the growth estimate $ \| \sqrt{\rho_{1}} \phi \|_{L^{2}}(t) \le Ct^{\frac{1}{2}} $. Moreover, we further show that
\begin{align}
t^{\frac{1}{4}}\| \sqrt{\rho_{1}} \mu \|_{L^{2}}(t) \le  t^{\frac{1}{4}}\|  \mu_{x} \|_{L^{2}}(t)  + \mathcal{O}(t^{\frac{1}{4}})  \le C.
\end{align}
From the above, we finally arrive at
\begin{align}
\| \phi_{x}\|_{L^{2}}^{2}(t) \le Ct^{\frac{1}{4}},
\end{align}
which directly implies \eqref{L-key-e}. For more details, see section \ref{sec3} (Step 2. Growth estimate).

The remainder of this paper is organized as follows.
Section~\ref{sec2} is devoted to the derivation of a priori estimates.
On this basis, the existence  and uniqueness of Theorem \ref{TH} are established
in Section~\ref{sec3}.

Throughout this paper, we use $C$, which may vary from place to place,  to denote a generic constant that may depend on $T$ but not on the lower bound of the density, unless we clearly specify.

\section{A priori estimates} \label{sec2}
In this section, we derive a priori estimates for smooth solutions
$(\rho,u,\phi)$ to system \eqref{nsch-1d}--\eqref{ic} on a finite time interval $(0,T)$.
The main goal is to obtain uniform bounds for the basic energy quantities and
higher-order derivatives of the solution in the presence of vacuum.

The estimates are obtained without imposing any initial compatibility conditions
and rely on suitable time-weighted energy inequalities to handle the degeneracy
induced by vanishing density. All the a priori bounds established in this section
are summarized in Proposition~\ref{p-estimate}, which serves as the foundation
for the proof of global existence and uniqueness of strong solutions in
Section~\ref{sec3}.

We begin by collecting several fundamental a priori estimates for smooth
solutions $(\rho,u,\phi)$ to system \eqref{nsch-1d}--\eqref{ic} on a finite time
interval $(0,T)$. These estimates are mostly available in the existing literature
and are recalled here for completeness, as they constitute the analytical
foundation of our subsequent arguments.
\begin{lemma}\label{L-0} {\rm (cf.\ Lemma 3.1 in \cite{D-L})}
Let $(\rho, u, \phi)$ be the smooth solutions to \eqref{nsch-1d}--\eqref{ic}, then it holds that
\begin{align}\label{0}
\mathcal{E}(t) + \int_{0}^{t} \left( \| u_{x} \|_{L^{2}}^{2} + \| \mu_{x}\|_{L^{2}}^{2} \right) {\rm d}s = \mathcal{E}(0),
\end{align}
where
\begin{align*}
\mathcal{E}(t) := \int_{I} \left( \frac{\rho u^{2}}{2} + \frac{\rho^{\gamma}}{\gamma-1} + \frac{\rho(\phi^{2}-1)^{2}}{4} + \frac{\phi^{2}_{x}}{2}  \right) (x, t)  {\rm d}x.
\end{align*}
\end{lemma}

%
\begin{lemma}{\rm (cf.\ Lemma 3.2 in \cite{D-L})}
Let $(\rho, u, \phi)$ be the smooth solutions to \eqref{nsch-1d}--\eqref{ic}, then it holds that
\begin{align} \label{t-phi-2}
\sup\limits_{0\le t \le T} \int_{I} \rho \phi^{2}  {\rm d}x + \int_{0}^{T} \int_{I} \left( \phi^{2} \phi_{x}^{2} + \frac{1}{\rho} \phi_{xx}^{2} \right) {\rm d}x {\rm d}t \le C(T).
\end{align}
\end{lemma}
\begin{remark}
	Observing that $\| \phi\|_{L^{\infty}(Q_{T})}$ can not be controlled by $\| \phi_{x}\|_{L^{\infty}(0,T; L^{2})}$, since $\phi$ satisfies Neumann boundary value condition. Consequently, an additional estimate is required to bound the concentration difference $\phi$.
\end{remark}
\begin{lemma} {\rm (cf.\ Lemma 2.2 in \cite{C-G} or Lemma 3.4 in \cite{D-L})}
Let $(\rho, u, \phi)$ be the smooth solutions to \eqref{nsch-1d}--\eqref{ic}, then it holds that
\begin{align} \label{phi-b}
 \| \phi \|_{L^{\infty}(Q_{T}) }\le C(T).
\end{align}
\end{lemma}

\begin{lemma}{\rm (cf.\ Lemma 2.3 in \cite{C-G})}
Let $(\rho, u, \phi)$ be the smooth solutions to \eqref{nsch-1d}--\eqref{ic}, then it holds that
\begin{align} \label{rho-b}
 \| \rho \|_{L^{\infty}( Q_{T}) }\le C(T).
\end{align}
\end{lemma}

Based on the above results, we establish the following lemma, which
is instrumental in the proof of the a priori estimates.

\begin{lemma}\label{L-5}
Let $(\rho, u, \phi)$ be the smooth solutions to \eqref{nsch-1d}--\eqref{ic}, then it holds that
\begin{align}\label{t-phi-2}
\int_{0}^{T} \left( \| \mu \|_{L^{2}}^{2} + \| \rho^{-1}\phi_{xx}\|_{L^{2}}^{2} + \| \phi_{xx}\|_{L^{2}}^{2} \right) {\rm d}t \le C(T).
\end{align}
\end{lemma}

\begin{proof}[\bf Proof]
Integrating \eqref{nsch-1d}$_{4}$ over $I$ and using the boundary value condition \eqref{bc} yields
\begin{align*}
\int_{I} \rho \mu  {\rm d}x = \int_{I} \rho (\phi^{3} - \phi)  {\rm d}x  \le  \left( \| \phi \|_{L^{\infty}}^{3} + \| \phi \|_{L^{\infty}} \right)  \int_{I} \rho   {\rm d}x \le C.
\end{align*}
Due to the above inequality, one has
\begin{align}\label{mu00}
| \mu(x,t) | & = \left| \mu(x,t) \int_{I} \rho(y,t)  {\rm d}y \right|
\nonumber \\
& \le \left| \int_{I}  \rho(y,t) \left( \mu(x,t) -\mu(y,t)  \right) {\rm d}y \right| + \left|  \int_{I} \rho(y,t) \mu(y,t)  {\rm d}y \right|
\nonumber \\
& \le  \int_{I}  \rho(y,t)  \left|  \int_{y}^{x} \mu_{\xi}(\xi,t)  {\rm d} \xi \right|  {\rm d}y + C
\nonumber \\
& \le  \int_{I}  \rho(y,t)   \int_{I} | \mu_{x} |  {\rm d} x  {\rm d}y + C
\nonumber \\
& =  \int_{I} | \mu_{x} |  {\rm d} x + C \le   \| \mu_{x} \|_{L^{2}} + C,
\end{align}
which together with \eqref{0} implies that
\begin{align} \label{mu-0-w}
\int_{0}^{T} \| \mu \|_{L^{2}}^{2}  {\rm d}t \le C \int_{0}^{T} \| \mu \|_{L^{\infty}}^{2}  {\rm d}t \le  C\int_{0}^{T} \left( \| \mu_{x} \|_{L^{2}}^{2} +1  \right)  {\rm d}t \le C.
\end{align}
It follows from \eqref{nsch-1d}$_{4}$, \eqref{phi-b} and \eqref{mu-0-w} that
\begin{align}\label{1-rho-phi-1}
\int_{0}^{T} \left\|  \rho^{-1} \phi_{xx} \right\|_{L^{2}}^{2}  {\rm d}t  \le C \int_{0}^{T} \| \mu \|_{L^{2}}^{2}  {\rm d}t + C \int_{0}^{T} \| \phi^{3} -\phi \|_{L^{2}}^{2}  {\rm d}t \le C.
\end{align}
Combining with \eqref{mu-0-w}, \eqref{1-rho-phi-1} and \eqref{rho-b} leads to Lemma \ref{L-5}.
\end{proof}

\begin{lemma}\label{L-9}
Let $(\rho, u, \phi)$ be smooth solutions to \eqref{nsch-1d}--\eqref{ic}. Then
\begin{align}\label{key}
\| \phi_{t} \|_{L^{\infty}} \le C \| \sqrt{\rho}\, \phi_{t} \|_{L^{2}}+ C \| \phi_{xt} \|_{L^{2}} .
\end{align}
\end{lemma}
\begin{proof}[\bf Proof]
The proof can be carried out in the same way as for \eqref{mu00} and is therefore omitted.
\end{proof}
\begin{lemma}\label{L-6}
Let $(\rho, u, \phi)$ be the smooth solutions to \eqref{nsch-1d}--\eqref{ic}, then it holds that
\begin{align} \label{e-6}
\sup\limits_{0\le t \le T} \left(t\left\|  \rho^{-\frac{1}{2}}\phi_{xx} \right\|_{L^{2}}^{2} \right) +  \int_{0}^{T} t \| \sqrt{\rho} \phi_{t} \|_{L^{2}}^{2}  {\rm d}t \le C(T).
\end{align}
\end{lemma}

\begin{proof}[\bf Proof]
By integration by parts and \eqref{nsch-1d}$_{4}$, we find
\begin{align}\label{f1}
\int_{I} \mu_{xx} \phi_{t} {\rm d}x = \int_{I} \mu \phi_{xxt} {\rm d}x = \int_{I} \left( -\frac{1}{\rho} \phi_{xx} + (\phi^{3} -\phi) \right) \phi_{xxt} {\rm d}x.
\end{align}
Testing \eqref{nsch-1d}$_{3}$ by $\phi_{t}$, integrating over $I$ and using \eqref{f1}, we have
\begin{align}\label{dt-1}
\frac{1}{2} \frac{\rm d}{{\rm d}t} \int_{I} \frac{1}{\rho} \phi_{xx}^{2}  {\rm d}x +  \int_{I} \rho \phi_{t}^{2}  {\rm d}x & = \frac{\rm d}{{\rm d}t} \int_{I} (\phi^{3} -\phi) \phi_{xx} {\rm d}x - \int_{I} (3\phi^{2} -1) \phi_{t} \phi_{xx} {\rm d}x
\nonumber \\
& \quad + \frac{1}{2} \int_{I} \left(\frac{1}{\rho} \right)_{t} \phi_{xx}^{2}  {\rm d}x - \int_{I} \rho u \phi_{x} \phi_{t}  {\rm d}x
\nonumber \\
& =:  \frac{\rm d}{{\rm d}t} \int_{I} (\phi^{3} -\phi) \phi_{xx} {\rm d}x + \sum\limits_{i=1}^{3} A_{i}.
\end{align}
We next estimate each term $A_{i} $ as follows. Direct calculation shows that
\begin{align*}
 &  A_{1}  \le C(\| \phi\|_{L^{\infty}}^{2} + 1) \| \sqrt{\rho} \phi_{t} \|_{L^{2}} \left\| \rho^{-\frac{1}{2}}  \phi_{xx} \right\|_{L^{2}} \le \frac{1}{4} \| \sqrt{\rho} \phi_{t} \|_{L^{2}}^{2} + C \left\| \rho^{-\frac{1}{2}}  \phi_{xx} \right\|_{L^{2}}^{2},
 \\[2mm]
 & A_{3}  \le C \| u \|_{L^{\infty}}  \| \phi_{x} \|_{L^{2}} \| \sqrt{\rho} \phi_{t} \|_{L^{2}} \le C \| u_{x} \|_{L^{2}}  \| \sqrt{\rho} \phi_{t} \|_{L^{2}} \le \frac{1}{4} \| \sqrt{\rho} \phi_{t} \|_{L^{2}}^{2} + C \| u_{x} \|_{L^{2}}^{2}.
\end{align*}
As for $A_{2}$, it follows from \eqref{nsch-1d}$_{1}$ and \eqref{nsch-1d}$_{4}$ that
\begin{align*}
 A_{2} & = \frac{1}{2} \int_{I} \frac{1}{\rho^{2}} (\rho u)_{x}  \phi_{xx}^{2}  {\rm d}x = - \frac{1}{2} \int_{I} \left( \frac{1}{\rho^{2}} \right)_{x} \rho u  \phi_{xx}^{2}  {\rm d}x -  \int_{I}  \frac{1}{\rho}  \phi_{xxx} u \phi_{xx} {\rm d}x
 \\
& =  \int_{I}  \frac{\rho_{x}}{\rho^{2}}  u  \phi_{xx}^{2}  {\rm d}x +  \int_{I}  \left[\mu_{x} -(3\phi^{2}-1) \phi_{x} + \left( \frac{1}{\rho} \right)_{x} \phi_{xx}  \right]  u \phi_{xx} {\rm d}x
\\
& =  \int_{I}  \left[\mu_{x} -(3\phi^{2}-1) \phi_{x}  \right]  u \phi_{xx} {\rm d}x
\\
& \le C \| u \|_{L^{\infty}} \left\|  \rho^{-\frac{1}{2}}\phi_{xx} \right\|_{L^{2}} \left(  \| \mu_{x} \|_{L^{2}} + \| 3\phi^{2}-1 \|_{L^{\infty}}  \| \phi_{x} \|_{L^{2}}  \right)
\\
& \le C\| u_{x} \|_{L^{2}}  \left\|  \rho^{-\frac{1}{2}}\phi_{xx} \right\|_{L^{2}}  \left(  \| \mu_{x} \|_{L^{2}} + 1  \right)
\\
& \le C \left(  \| \mu_{x} \|_{L^{2}}^{2} + 1  \right) + C \| u_{x} \|_{L^{2}}^{2} \left\|  \rho^{-\frac{1}{2}}\phi_{xx} \right\|_{L^{2}}^{2}.
\end{align*}
Substituting $A_{1}$--$A_{3}$ into \eqref{dt-1} and multiplying \eqref{dt-1} by $t$, we get
\begin{align}\label{dt-1.5}
& \frac{\rm d}{{\rm d}t} \left( t \left\|  \rho^{-\frac{1}{2}}\phi_{xx} \right\|_{L^{2}}^{2} \right)  +  t\| \sqrt{\rho} \phi_{t} \|_{L^{2}}^{2}
 \nonumber \\
 \le& \frac{\rm d}{{\rm d}t} \left( t\int_{I} (\phi^{3} -\phi) \phi_{xx} {\rm d}x \right) +  \left\|  \rho^{-\frac{1}{2}}\phi_{xx} \right\|_{L^{2}}^{2}-\int_{I} (\phi^{3} -\phi) \phi_{xx} {\rm d}x
\nonumber \\
&  + C t \left(  \| \mu_{x} \|_{L^{2}}^{2}+ \| u_{x} \|_{L^{2}}^{2} + 1  \right) + C (\| u_{x} \|_{L^{2}}^{2} + 1) \left( t \left\|  \rho^{-\frac{1}{2}}\phi_{xx} \right\|_{L^{2}}^{2} \right)
 \nonumber \\
 \le& \frac{\rm d}{{\rm d}t} \left( t\int_{I} (\phi^{3} -\phi) \phi_{xx} {\rm d}x \right) + C (\| u_{x} \|_{L^{2}}^{2} + 1) \left( t \left\|  \rho^{-\frac{1}{2}}\phi_{xx} \right\|_{L^{2}}^{2} \right)
\nonumber \\
&  + C(T) \left(  \| \mu_{x} \|_{L^{2}}^{2}+ \| u_{x} \|_{L^{2}}^{2} + \left\|  \rho^{-1}\phi_{xx} \right\|_{L^{2}}^{2} + \|\phi_{xx}\|_{L^{2}}^{2}  +1 \right).
\end{align}
Integrating \eqref{dt-1.5} over $[0,t]$, using \eqref{0} and \eqref{t-phi-2}, we obtain
\begin{align*}
& t \left\|  \rho^{-\frac{1}{2}}\phi_{xx}  \right\|_{L^{2}}^{2}   + \int_{0}^{t} t \| \sqrt{\rho} \phi_{t} \|_{L^{2}}^{2}  {\rm d}s
\\
\le&   t \int_{I} [(\phi^{3} -\phi) \phi_{xx}] (x,t) {\rm d}x +  C \int_{0}^{t} \left(\| u_{x} \|_{L^{2}}^{2} + 1 \right) \left(s\left\|  \rho^{-\frac{1}{2}}\phi_{xx} \right\|_{L^{2}}^{2} \right) {\rm d}s+ C
\\
 \le & \frac{1}{2} \left( t \left\|  \rho^{-\frac{1}{2}}\phi_{xx} \right\|_{L^{2}}^{2} \right)
 +  C \int_{0}^{t} \left(\| u_{x} \|_{L^{2}}^{2} + 1 \right) \left(s\left\|  \rho^{-\frac{1}{2}}\phi_{xx} \right\|_{L^{2}}^{2} \right) {\rm d}s + C,
\end{align*}
which, together with \eqref{0} and Gronwall's inequality, implies Lemma \ref{L-6}.
\end{proof}

\begin{lemma}\label{L-7}
Let $(\rho, u, \phi)$ be the smooth solutions to \eqref{nsch-1d}--\eqref{ic}, then it holds that
\begin{align}\label{u-x}
\sup\limits_{0\le t \le T} \| u_{x}\|_{L^{2}}^{2} +  \int_{0}^{T} \|  \sqrt{\rho} u_{t} \|_{L^{2}}^{2}  {\rm d}t \le C(T).
\end{align}
\end{lemma}

\begin{proof}[\bf Proof]
Testing \eqref{nsch-1d}$_{2}$ by $u_{t}$ and integrating over $I$ by parts, one has
\begin{align}\label{dt-2}
\frac{1}{2} \frac{\rm d}{{\rm d}t} \| u_{x}\|_{L^{2}}^{2} +  \|  \sqrt{\rho} u_{t} \|_{L^{2}}^{2}
 & =  \frac{\rm d}{{\rm d}t} \int_{I} \rho^{\gamma} u_{x}  {\rm d}x - \int_{I} (\rho^{\gamma})_{t} u_{x}  {\rm d}x  - \int_{I} \rho u u_{x} u_{t} {\rm d}x - \int_{I} \phi_{x} \phi_{xx} u_{t} {\rm d}x
\nonumber \\
& =:  \frac{\rm d}{{\rm d}t} \int_{I} \rho^{\gamma} u_{x}  {\rm d}x + \sum\limits_{i=1}^{3} B_{i}
\end{align}
We next estimate each term $B_{i} $ as follows. Firstly, using \eqref{nsch-1d}$_{1}$, \eqref{0}, \eqref{rho-b} and integration by parts formula, one infers that
\begin{align*}
 B_{1} & = \int_{I} \gamma \rho^{\gamma-1} (\rho u)_{x} u_{x}   {\rm d}x =   \int_{I} \gamma \rho^{\gamma}  u_{x}^{2}   {\rm d}x + \int_{I} (\rho^{\gamma})_{x}  u u_{x}   {\rm d}x
\\
& =  (\gamma -1) \int_{I} \rho^{\gamma}  u_{x}^{2}   {\rm d}x -  \int_{I} \rho^{\gamma}  u u_{xx}   {\rm d}x
\\
& =  (\gamma -1) \int_{I} \rho^{\gamma}  u_{x}^{2}   {\rm d}x -  \int_{I} \rho^{\gamma}  u \left[ \rho u_{t} + \rho u u_{x} +  \phi_{x} \phi_{xx}\right]   {\rm d}x + \frac{1}{2} \int_{I} \rho^{2\gamma}  u_{x} {\rm d}x
\\
& \le C \| u_{x}\|_{L^{2}}^{2} + C \| u \|_{L^{\infty}} \left( \| \sqrt{\rho}u_{t} \|_{L^{2}} +  \| u \|_{L^{2}} \| u_{x} \|_{L^{2}} + \| \phi_{x} \|_{L^{2}} \| \phi_{xx} \|_{L^{2}} \right) + C \| u_{x}\|_{L^{2}}
\\
& \le  \frac{1}{6} \| \sqrt{\rho}u_{t} \|_{L^{2}}^{2}  + C \| u_{x}\|_{L^{2}}^{2} + C \left( \| \phi_{xx} \|_{L^{2}}^{2} + 1 \right).
\end{align*}
Next, direct calculation shows that
\begin{align*}
& B_{2} \le C \| u \|_{L^{\infty}} \| u_{x} \|_{L^{2}} \| \sqrt{\rho}u_{t} \|_{L^{2}} \le \frac{1}{6} \| \sqrt{\rho}u_{t} \|_{L^{2}}^{2} + C \| u_{x} \|_{L^{2}}^{4}.
\end{align*}
Finally, it follows from \eqref{nsch-1d}$_{4}$, \eqref{0} and \eqref{phi-b} that
\begin{align*}
& B_{3} = \int_{I} \left(\rho \mu -\rho (\phi^{3}-\phi) \right)  \phi_{x} u_{t} {\rm d}x
\\
& \le C\left(  \| \mu \|_{L^{\infty}} +  \| \phi^{3} \|_{L^{\infty}} + \| \phi \|_{L^{\infty}} \right)  \| \phi_{x} \|_{L^{2}}  \| \sqrt{\rho}u_{t} \|_{L^{2}}
\\
& \le  C\left(  \| \mu \|_{H^{1}} +  1 \right)   \| \sqrt{\rho}u_{t} \|_{L^{2}}
 \le \frac{1}{6} \| \sqrt{\rho}u_{t} \|_{L^{2}}^{2} +  C \| \mu \|_{H^{1}}^{2} + C.
\end{align*}
Substituting $B_{1}$--$B_{3}$ into \eqref{dt-2}, integrating it over $[0,t]$ and using \eqref{0}, \eqref{rho-b} and \eqref{t-phi-2}, we arrive at
\begin{align*}
&  \| u_{x}(t)\|_{L^{2}}^{2} +  \int_{0}^{t}\|  \sqrt{\rho} u_{t} \|_{L^{2}}^{2} {\rm d}s
 \nonumber \\
 & \le  \| u_{x}(0)\|_{L^{2}}^{2} - \int_{I} (\rho^{\gamma} u_{x} )(x, 0)  {\rm d}x + \int_{I} \rho^{\gamma} u_{x}(x, t)  {\rm d}x
 \nonumber \\
 & \quad + C \int_{0}^{t} \| u_{x}\|_{L^{2}}^{2} \| u_{x}\|_{L^{2}}^{2} {\rm d}s + C
\nonumber \\
& \le C  ( \| u_{0}\|_{H^{1}}^{2} + \| \rho_{0}\|_{H^{1}}^{\gamma}+ 1) + \frac{1}{2}\| u_{x}(t)\|_{L^{2}}^{2} + C \int_{0}^{t} \| u_{x}\|_{L^{2}}^{2} \| u_{x}\|_{L^{2}}^{2} {\rm d}s + C,
 \end{align*}
which, together with \eqref{0} and Gronwall's inequality, completes the proof of Lemma \ref{L-7}.
\end{proof}

\begin{lemma}\label{L-8}
Let $(\rho, u, \phi)$ be the smooth solutions to \eqref{nsch-1d}--\eqref{ic}, then it holds that
\begin{align}\label{1-rho-b}
\sup\limits_{0\le t \le T} \left(  \| \rho_{x}\|_{L^{2}}^{2}  + \| \rho_{t}\|_{L^{2}}^{2}  \right) + \int_{0}^{T} \| u_{xx} \|_{L^{2}} {\rm d}t \le C(T).
\end{align}
\end{lemma}

\begin{proof}[\bf Proof]
Differentiating \eqref{nsch-1d}$_{1}$ with respect to $x$, testing the resultant equation by $\rho_{x}$, and then integrating over $I$ by parts, we have
\begin{align}\label{dt-3}
\frac{1}{2} \frac{\rm d}{{\rm d}t}  \| \rho_{x}\|_{L^{2}}^{2} &  = -\frac{3}{2} \int_{I} \rho_{x}^{2} u_{x} {\rm d}x - \int_{I}  \rho \rho_{x} \left[ \rho u_{t} + \rho u u_{x} + (\rho^{\gamma})_{x} + \phi_{x} \phi_{xx} \right] {\rm d}x
\nonumber \\
& \le C \| u_{x}\|_{L^{\infty}} \| \rho_{x}\|_{L^{2}}^{2} + C \| \rho_{x}\|_{L^{2}} \left( \| \sqrt{\rho}u_{t}\|_{L^{2}} + \| u\|_{L^{\infty}}\| u_{x}\|_{L^{2}}  \right)
\nonumber \\
& \quad  + C \| \rho_{x}\|_{L^{2}} \left( \| \rho_{x}\|_{L^{2}} +  \| \phi_{x}\|_{L^{\infty}}\| \phi_{xx}\|_{L^{2}} \right)
\nonumber \\
&  \le C \left(  \| u_{x}\|_{L^{\infty}} + \| \phi_{xx}\|_{L^{2}}^{2} + 1 \right) \| \rho_{x}\|_{L^{2}}^{2} + C  \left(  \| \sqrt{\rho}u_{t}\|_{L^{2}}^{2} +  \| \phi_{xx}\|_{L^{2}}^{2} + 1\right).
\end{align}
To proceed, we note that
\begin{align} \label{u-wq}
\| u_{x}\|_{L^{\infty}} & \le \| (u_{x} - \rho^{\gamma})\|_{L^{\infty}} + \| \rho^{\gamma}\|_{L^{\infty}}
\nonumber \\
& \le C \| (u_{x} - \rho^{\gamma})\|_{L^{1}} + C\| (u_{x} - \rho^{\gamma})_{x}\|_{L^{1}} + C
\nonumber \\
& \le C \|  \rho u_{t} + \rho u u_{x} + \phi_{x} \phi_{xx}   \|_{L^{1}} + C
\nonumber \\
& \le C \left(  \| \sqrt{\rho}u_{t}\|_{L^{2}} +  \| u\|_{L^{2}}  \| u_{x}\|_{L^{2}}  +  \| \phi_{x}\|_{L^{2}}  \| \phi_{xx}\|_{L^{2}} + 1\right)
\nonumber \\
& \le C \left(  \| \sqrt{\rho}u_{t}\|_{L^{2}} +  \| \phi_{xx}\|_{L^{2}} + 1\right),
\end{align}
where we have used \eqref{0}, \eqref{rho-b}, \eqref{u-x} and Sobolev inequality. Substituting \eqref{u-wq} into \eqref{dt-3}, using \eqref{t-phi-2}, \eqref{u-x} and Gronwall's inequality, one deduces
\begin{align*}
\sup\limits_{0\le t \le T} \| \rho_{x}\|_{L^{2}}^{2}  \le C.
\end{align*}
This together with \eqref{nsch-1d}$_{1}$, \eqref{rho-b} and \eqref{u-x} shows
\begin{align*}
\sup\limits_{0\le t \le T} \| \rho_{t}\|_{L^{2}}^{2}  \le C.
\end{align*}
It follows from \eqref{nsch-1d}$_{2}$ that
\begin{align}\label{t-uxx}
\int_{0}^{T} \| u_{xx} \|_{L^{2}} {\rm d}t & \le C \int_{0}^{T} \left( \| \sqrt{\rho} u_{t} \|_{L^{2}} + \| u\|_{L^{\infty}} \| u_{x}\|_{L^{2}} +   \| \rho_{x}\|_{L^{2}} + \| \phi_{x}\|_{L^{\infty}} \| \phi_{xx}\|_{L^{2}} \right) {\rm d}t
\\
& \le C\int_{0}^{T} \left(1+ \| \sqrt{\rho} u_{t} \|_{L^{2}}^{2} + \| u_{x}\|_{L^{2}}^{2} +   \| \rho_{x}\|_{L^{2}}^{2} +  \| \phi_{xx}\|_{L^{2}}^{2} \right) {\rm d}t \le C.
\end{align}
Thus, we complete the proof of Lemma \ref{L-8}.
\end{proof}


\begin{lemma}\label{L-10}
Let $(\rho, u, \phi)$ be the smooth solutions to \eqref{nsch-1d}--\eqref{ic}, then it holds that
\begin{align}\label{mu-1}
\sup\limits_{0\le t \le T} \left( t \| \mu \|_{H^{1}}^{2} + t\| \phi_{xxx} \|_{L^{2}}^{2} \right) + \int_{0}^{T} \left( t\| \phi_{t} \|_{H^{1}}^{2} +  t^{\frac{1}{2}}\| \sqrt{\rho} \phi_{t} \|_{L^{2}}^{2} \right)  {\rm d}t  \le C(T).
\end{align}
\end{lemma}

\begin{proof}[\bf Proof]
Testing \eqref{nsch-1d}$_{3}$ by $-\mu_{t}$, using \eqref{nsch-1d}$_{4}$ and integrating over $I$ by parts, one has
\begin{align}\label{dt-5}
 \frac{1}{2} \frac{\rm d}{{\rm d}t}  \| \mu_{x}\|_{L^{2}}^{2}  + \| \phi_{xt} \|_{L^{2}}^{2} & =  \int_{I} \rho \phi_{t} \left( \frac{1}{\rho} \right)_{t} \phi_{xx}  {\rm d}x   - \int_{I}  (3\phi^{2}-1)\rho \phi_{t}^{2}  {\rm d}x
\nonumber \\
& \quad + \int_{I} \rho u \phi_{x} \left( \frac{1}{\rho} \phi_{xx} \right)_{t} {\rm d}x - \int_{I}  (3\phi^{2}-1) \rho u \phi_{x} \phi_{t}  {\rm d}x
\nonumber \\
& =: \sum \limits_{i=1}^{4} G_{i}.
\end{align}
We now estimate each term on the right hand side of \eqref{dt-5} as follows. Firstly, it follows from \eqref{1-rho-b} and \eqref{key} that
\begin{align*}
G_{1} & = - \int_{I} \frac{1}{\rho}  \rho_{t} \phi_{t} \phi_{xx}  {\rm d}x \le \| \phi_{t}\|_{L^{\infty}} \| \rho^{-1} \phi_{xx} \|_{L^{2}}  \| \rho_{t} \|_{L^{2}}
\\
& \le C\| \phi_{t}\|_{H^{1}} \| \rho^{-1} \phi_{xx} \|_{L^{2}} \le C \left( \| \sqrt{\rho}\phi_{t}\|_{L^{2}} + \| \phi_{xt}\|_{L^{2}}\right)  \| \rho^{-1} \phi_{xx} \|_{L^{2}}
\\
& \le \frac{1}{4} \| \phi_{xt}\|_{L^{2}}^{2} + C\| \sqrt{\rho}\phi_{t}\|_{L^{2}}^{2} + C\| \rho^{-1} \phi_{xx} \|_{L^{2}}^{2}.
\end{align*}
Then, noting $-(3\phi^{2}-1 ) \le  1$ and $\rho \phi_{t}^{2} \ge 0$, we have
\begin{align*}
G_{2} =  \int_{I} - (3\phi^{2}-1)\rho \phi_{t}^{2}  {\rm d}x  \le \int_{I}  \rho \phi_{t}^{2}  {\rm d}x = \| \sqrt{\rho}\phi_{t}\|_{L^{2}}^{2}.
\end{align*}
Next, using integration by parts formula and \eqref{u-x}, one has
\begin{align*}
G_{3} & =  - \int_{I}  \frac{1}{\rho} \rho_{t} u \phi_{x}  \phi_{xx}  {\rm d}x + \int_{I}  u \phi_{x}  \phi_{xxt}  {\rm d}x
\\
& = - \int_{I}  \frac{1}{\rho} \rho_{t} u \phi_{x}  \phi_{xx}  {\rm d}x - \int_{I}  u_{x} \phi_{x}  \phi_{xt}  {\rm d}x -\int_{I}  u \phi_{xx}  \phi_{xt}  {\rm d}x
\\
& \le  \| u \|_{L^{\infty}} \| \phi_{x}\|_{L^{\infty}} \| \rho^{-1} \phi_{xx} \|_{L^{2}}  \| \rho_{t} \|_{L^{2}} +  \| u_{x} \|_{L^{2}}  \| \phi_{x}\|_{L^{\infty}}   \| \phi_{xt} \|_{L^{2}}
\\
& \quad + \| u \|_{L^{\infty}}  \| \phi_{xx}\|_{L^{2}}   \| \phi_{xt} \|_{L^{2}}
\\
& \le C\| u_{x} \|_{L^{2}} \| \phi_{xx}\|_{L^{2}} \| \rho^{-1} \phi_{xx} \|_{L^{2}}  + C\| u_{x} \|_{L^{2}} \| \phi_{xx}\|_{L^{2}} \| \phi_{xt} \|_{L^{2}}
\\
& \le \frac{1}{4} \| \phi_{xt}\|_{L^{2}}^{2}  + C\| \phi_{xx}\|_{L^{2}}^{2} + C\| \rho^{-1} \phi_{xx} \|_{L^{2}}^{2}.
\end{align*}
Finally, due to \eqref{0}, \eqref{u-x} and $-(3\phi^{2}-1 ) \le  1$, there holds
\begin{align*}
G_{4} & \le  \int_{I}  \rho u \phi_{x} \phi_{t}  {\rm d}x \le C  \| u \|_{L^{\infty}} \| \sqrt{\rho}\phi_{t}\|_{L^{2}}   \| \phi_{x} \|_{L^{2}}
\\
& \le C \| u_{x} \|_{L^{2}} \| \sqrt{\rho}\phi_{t}\|_{L^{2}} \le C\| \sqrt{\rho}\phi_{t}\|_{L^{2}}^{2} + C.
\end{align*}
Substituting $G_{1}$--$G_{4}$ into \eqref{dt-5}, and then multiplying the result by $t$, yields
\begin{align*}
 \frac{\rm d}{{\rm d}t} \left( t \| \mu_{x}\|_{L^{2}}^{2} \right) + t \| \phi_{xt} \|_{L^{2}}^{2} \le  \| \mu_{x}\|_{L^{2}}^{2} + C t\| \sqrt{\rho}\phi_{t}\|_{L^{2}}^{2}  + C(T)\left(\| \phi_{xx}\|_{L^{2}}^{2} +\| \rho^{-1} \phi_{xx} \|_{L^{2}}^{2} + 1\right).
\end{align*}
Integrating the above inequality over $[0,t]$ and then using \eqref{t-phi-2} and \eqref{e-6}, we have
\begin{align}\label{mu11}
\sup\limits_{0\le t \le T} \left( t\| \mu_{x} \|_{L^{2}}^{2} \right)  + \int_{0}^{T} t\| \phi_{xt} \|_{L^{2}}^{2}  {\rm d}t  \le  C.
\end{align}
Thus, we obtain from \eqref{nsch-1d}$_{3}$ that
\begin{align}\label{pphi-pro}
& \int_{0}^{T} t^{\frac{1}{2}}\| \sqrt{\rho} \phi_{t} \|_{L^{2}}^{2} {\rm d}t
\nonumber \\
 =& - \int_{0}^{T} t^{\frac{1}{2}} \int_{\Omega} \rho u \phi_{x} \phi_{t} {\rm d}x {\rm d}t  - \int_{0}^{T} t^{\frac{1}{2}} \int_{\Omega} \mu_{x}  \phi_{xt} {\rm d}x {\rm d}t
\nonumber \\
 \le& \frac{1}{2}\int_{0}^{T} t^{\frac{1}{2}}\| \sqrt{\rho} \phi_{t} \|_{L^{2}}^{2} {\rm d}t + C \int_{0}^{T} t^{\frac{1}{2}} \| u \|_{L^{\infty}}^{2} \| \phi_{x} \|_{L^{2}}^{2} {\rm d}t
 + \int_{0}^{T}  \| \mu_{x} \|_{L^{2}}^{2} {\rm d}t +  \int_{0}^{T} t \| \phi_{xt} \|_{L^{2}}^{2} {\rm d}t
 \nonumber \\
 \le& \frac{1}{2}\int_{0}^{T} t^{\frac{1}{2}}\| \sqrt{\rho} \phi_{t} \|_{L^{2}}^{2} {\rm d}t + C(T).
\end{align}
It follows from \eqref{nsch-1d}$_{4}$,  \eqref{0}, \eqref{phi-b}, \eqref{rho-b}, \eqref{mu00},  \eqref{1-rho-b} and \eqref{mu11} that
\begin{align*}
\sqrt{t}\| \phi_{xxx} \|_{L^{2}}   \le   \sqrt{t}\| \rho_{x} \|_{L^{2}}  \| \mu \|_{L^{\infty}}  +  \sqrt{t} \| \rho \|_{L^{\infty}}  \| \mu_{x} \|_{L^{2}} + C(T)\| \phi_{x} \|_{L^{2}}  + C(T) \| \rho_{x} \|_{L^{2}} \le  C(T).
\end{align*}
This, together with \eqref{mu00}, \eqref{key}, \eqref{mu11} and \eqref{pphi-pro}, completes the proof of Lemma \ref{L-10}.
\end{proof}

\begin{lemma}\label{L-11}
Let $(\rho, u, \phi)$ be the smooth solutions to \eqref{nsch-1d}--\eqref{ic}, then it holds that
\begin{align}\label{e-11}
\sup\limits_{0\le t \le T} \left(  t^{2}\|  \sqrt{\rho}\phi_{t} \|_{L^{2}}^{2} + t^{2}\|  \phi_{t} \|_{H^{1}}^{2} + t^{2} \|  \mu_{xx}\|_{L^{2}}^{2} \right) + \int_{0}^{T} t^{2} \left( \|  \mu_{t} \|_{H^{1}}^{2}  +  \|  \phi_{xxt} \|_{L^{2}}^{2} \right) {\rm d}t  \le C(T).
\end{align}
\end{lemma}

\begin{proof}[\bf Proof]
The proof will be completed by several steps.

{\bf Step 1. Estimate for $\| \sqrt{\rho}\phi_{t}\|_{L^{2}}$.}

Differentiating \eqref{nsch-1d}$_{3}$ with respect to $t$ yields
\begin{align}\label{test-1}
\rho \phi_{tt} + \rho_{t} \phi_{t} + \rho_{t} u\phi_{x} + \rho u_{t}\phi_{x} + \rho u\phi_{xt}  = \mu_{xxt}.
\end{align}
Testing \eqref{test-1} by $\phi_{t}$, then integrating over $I$ by parts, using \eqref{nsch-1d}$_{1}$, we have
\begin{align*}
 \frac{1}{2} \frac{\rm d}{{\rm d}t} \| \sqrt{\rho} \phi_{t} \|_{L^{2}}^{2} & =  -2 \int_{I} \rho u \phi_{t} \phi_{xt}  {\rm d}x  + \int_{I} \rho u_{x} u \phi_{x} \phi_{t}  {\rm d}x + \int_{I} \rho_{x} u^{2} \phi_{x} \phi_{t}  {\rm d}x
 \nonumber \\
& \quad - \int_{I} \rho u_{t} \phi_{x} \phi_{t}  {\rm d}x - \int_{I} \mu_{xt} \phi_{xt}  {\rm d}x
\nonumber \\
& \le C \| u \|_{L^{\infty}} \| \sqrt{\rho} \phi_{t} \|_{L^{2}} \| \phi_{xt} \|_{L^{2}} + \| u \|_{L^{\infty}} \| u_{x} \|_{L^{2}} \| \sqrt{\rho} \phi_{t} \|_{L^{2}} \| \phi_{x} \|_{L^{\infty}}
\nonumber \\
& \quad + \| \rho_{x} \|_{L^{2}}  \| u \|_{L^{\infty}}^{2} \|  \phi_{t} \|_{L^{\infty}} \| \phi_{x} \|_{L^{2}} + \| \mu_{xt} \|_{L^{2}}  \| \phi_{xt} \|_{L^{2}}
\nonumber \\
& \quad  + \| \sqrt{\rho} u_{t} \|_{L^{2}} \| \sqrt{\rho} \phi_{t} \|_{L^{2}} \| \phi_{x} \|_{L^{\infty}}
\nonumber \\
& \le  \frac{1}{12}  \| \mu_{xt} \|_{L^{2}}^{2} +   C \| \phi_{xx} \|_{L^{2}}^{2}  \left( \| \sqrt{\rho} \phi_{t} \|_{L^{2}}^{2} +   \| \phi_{xt} \|_{L^{2}}^{2} \right) +C \| \sqrt{\rho} u_{t} \|_{L^{2}}^{2} + C,
\end{align*}
which implies that
\begin{align}\label{dt-5.5}
 \frac{1}{2} \frac{\rm d}{{\rm d}t} \left(t^{2}\| \sqrt{\rho} \phi_{t} \|_{L^{2}}^{2} \right) & \le  \frac{1}{12} t^{2}\| \mu_{xt} \|_{L^{2}}^{2} + 2 t \| \sqrt{\rho} \phi_{t} \|_{L^{2}}^{2} + C(T)\| \sqrt{\rho} u_{t} \|_{L^{2}}^{2} + C(T)
 \nonumber \\
 & \quad +   C   \| \phi_{xx} \|_{L^{2}}^{2} \left( t^{2}\| \sqrt{\rho} \phi_{t} \|_{L^{2}}^{2} +  t^{2}\| \phi_{xt} \|_{L^{2}}^{2} \right).
\end{align}

{\bf Step 2. Estimate for $\| \phi_{xt}\|_{L^{2}}$.}

Differentiating \eqref{nsch-1d}$_{4}$ with respect to $t$ yields
\begin{align}\label{test-2}
\rho \mu_{t} + \rho_{t} \mu = -\phi_{xxt} + \rho_{t} (\phi^{3}-\phi) + \rho (3\phi^{2}-1)\phi_{t}.
\end{align}
Testing \eqref{test-1} by $\mu_{t}$, then integrating over $I$ by parts, using \eqref{test-2}, we get
\begin{align}\label{dt-6}
& \frac{1}{2} \frac{\rm d}{{\rm d}t} \|  \phi_{xt} \|_{L^{2}}^{2} + \|  \mu_{xt} \|_{L^{2}}^{2}
\nonumber \\
& = - \int_{I} \rho_{t} \phi_{t} \mu_{t}  {\rm d}x - \int_{I} \rho_{t} u \phi_{x} \mu_{t}  {\rm d}x - \int_{I} \rho u_{t} \phi_{x} \mu_{t}  {\rm d}x - \int_{I} \rho u \phi_{xt} \mu_{t}  {\rm d}x
\nonumber \\
& \quad + \int_{I} \phi_{tt} \rho_{t} [ \mu- (\phi^{3}-\phi)  ]  {\rm d}x - \int_{I} \phi_{tt} \rho (3\phi^{2}-1) \phi_{t}  {\rm d}x
\nonumber \\
& =: \sum\limits_{i=1}^{6} H_{i}.
\end{align}
We are going to estimate each term on the right hand side of \eqref{dt-6} as follows. For $i=1,\cdots,4$, direct calculation shows that
\begin{align*}
H_{1} & \le \| \rho_{t} \|_{L^{2}}\| \phi_{t} \|_{L^{2}} \| \mu_{t} \|_{L^{\infty}}  \le \varepsilon \| \mu_{t} \|_{L^{\infty}}^{2} + C \left( \| \sqrt{\rho} \phi_{t} \|_{L^{2}}^{2} +   \| \phi_{xt} \|_{L^{2}}^{2} \right),
\\[1em]
H_{2} &\le  \| \rho_{t} \|_{L^{2}}   \| u \|_{L^{\infty}} \| \phi_{x} \|_{L^{2}} \| \mu_{t} \|_{L^{\infty}} \le \varepsilon \| \mu_{t} \|_{L^{\infty}}^{2} + C,
\\[1em]
H_{3} &\le C \| \sqrt{\rho} u_{t} \|_{L^{2}}  \| \phi_{x} \|_{L^{2}} \| \mu_{t} \|_{L^{\infty}} \le \varepsilon \| \mu_{t} \|_{L^{\infty}}^{2} + C\| \sqrt{\rho} u_{t} \|_{L^{2}}^{2},
\\[1em]
H_{4} &\le C \| u \|_{L^{2}}  \| \phi_{xt} \|_{L^{2}} \| \mu_{t} \|_{L^{\infty}} \le \varepsilon \| \mu_{t} \|_{L^{\infty}}^{2} + C\| \phi_{xt} \|_{L^{2}}^{2},
\end{align*}
where $\varepsilon$ is sufficiently small and will be specified later. It follows from \eqref{nsch-1d}$_{1}$ and integration by parts formula that
\begin{align*}
H_{5} & = -  \int_{I} \phi_{tt} (\rho u)_{x} [ \mu- (\phi^{3}-\phi)  ]  {\rm d}x
\\
& = \int_{I} \phi_{xtt} \rho u [ \mu- (\phi^{3}-\phi)  ]  {\rm d}x + \int_{I} \phi_{tt} \rho u [ \mu_{x}- (3\phi^{2}-1)\phi_{x}  ]  {\rm d}x
\\
& = \frac{\rm d}{{\rm d}t} \int_{I} \phi_{xt} \rho u [ \mu- (\phi^{3}-\phi)  ]  {\rm d}x - \int_{I} \phi_{xt} \rho_{t} u [ \mu- (\phi^{3}-\phi)  ]  {\rm d}x
\\
& \quad - \int_{I} \phi_{xt} \rho u_{t} [ \mu- (\phi^{3}-\phi)  ]  {\rm d}x - \int_{I} \phi_{xt} \rho u [ \mu_{t}- (3\phi^{2}-1) \phi_{t} ]  {\rm d}x
\\
& \quad +\frac{\rm d}{{\rm d}t} \int_{I} \phi_{t} \rho u [ \mu_{x}- (3\phi^{2}-1)\phi_{x}  ]  {\rm d}x - \int_{I} \phi_{t} \rho_{t} u [ \mu_{x}- (3\phi^{2}-1)\phi_{x}  ]  {\rm d}x
\\
& \quad -  \int_{I} \phi_{t} \rho u_{t} [ \mu_{x}- (3\phi^{2}-1)\phi_{x}  ]  {\rm d}x -\int_{I} \phi_{t} \rho u [ \mu_{xt}- (3\phi^{2}-1)\phi_{xt} - 6\phi \phi_{t} \phi_{x} ]  {\rm d}x
\\
& \le \frac{\rm d}{{\rm d}t} \int_{I} \phi_{xt} \rho u [ \mu- (\phi^{3}-\phi)  ]  {\rm d}x + \frac{\rm d}{{\rm d}t} \int_{I} \phi_{t} \rho u [ \mu_{x}- (3\phi^{2}-1)\phi_{x}  ]  {\rm d}x
\\
& \quad + C\| \phi_{xt} \|_{L^{2}} \left( \| \rho_{t} \|_{L^{2}} \| u \|_{L^{\infty}} + \| \sqrt{\rho} u_{t} \|_{L^{2}} \right) \left( \| \mu \|_{L^{\infty}} + \| \phi^{3} - \phi \|_{L^{\infty}} \right)
\\
& \quad + C\| \phi_{xt} \|_{L^{2}} \| u \|_{L^{2}} \left( \| \mu_{t} \|_{L^{\infty}} +  \| 3\phi^{2} - 1\|_{L^{\infty}} \| \phi_{t} \|_{L^{\infty}} \right)
\\
& \quad + C \| \phi_{t} \|_{L^{\infty}}  \left( \| \rho_{t} \|_{L^{2}} \| u \|_{L^{\infty}} + \| \sqrt{\rho} u_{t} \|_{L^{2}} \right) \left( \| \mu_{x} \|_{L^{2}} +  \| 3\phi^{2} - 1\|_{L^{\infty}} \| \phi_{x} \|_{L^{2}} \right)
\\
& \quad + C \| \phi_{t} \|_{L^{\infty}} \| u \|_{L^{2}} \left( \| \mu_{xt} \|_{L^{2}} +  \| 3\phi^{2} - 1\|_{L^{\infty}} \| \phi_{xt} \|_{L^{2}}  +  \| \phi\|_{L^{\infty}}  \| \phi_{x} \|_{L^{2}} \| \phi_{t} \|_{L^{\infty}} \right)
\\
& \le \frac{\rm d}{{\rm d}t} \int_{I} \phi_{xt} \rho u [ \mu- (\phi^{3}-\phi)  ]  {\rm d}x + \frac{\rm d}{{\rm d}t} \int_{I} \phi_{t} \rho u [ \mu_{x}- (3\phi^{2}-1)\phi_{x}  ]  {\rm d}x +C\| \mu \|_{H^{1}}^{2}
\\
& \quad + \frac{1}{6} \| \mu_{xt} \|_{L^{2}}^{2} + \varepsilon \| \mu_{t} \|_{L^{\infty}}^{2} +  C(1+ \| \sqrt{\rho} u_{t} \|_{L^{2}}^{2} )\left(  \| \sqrt{\rho} \phi_{t} \|_{L^{2}}^{2} + \| \phi_{xt} \|_{L^{2}}^{2} \right)   + C.
\end{align*}
Similarly, one has
\begin{align*}
H_{6} & = -\frac{1}{2} \frac{\rm d}{{\rm d}t} \int_{I} \phi_{t}^{2} \rho   (3\phi^{2}-1)    {\rm d}x - \frac{1}{2} \int_{I} \phi_{t}^{2} \rho_{t}   (3\phi^{2}-1)    {\rm d}x - 3 \int_{I} \phi_{t}^{3} \rho \phi   {\rm d}x
\\
& \le -\frac{1}{2} \frac{\rm d}{{\rm d}t} \int_{I} \phi_{t}^{2} \rho   (3\phi^{2}-1)    {\rm d}x  + C  \| \phi_{t} \|_{L^{\infty}}^{2}  \| \rho_{t} \|_{L^{2}} \| 3\phi^{2} - 1\|_{L^{2}}
\\
& \quad + C \| \sqrt{\rho}\phi_{t} \|_{L^{2}}^{2}  \| \phi_{t} \|_{L^{\infty}} \| \phi\|_{L^{\infty}}
\\
& \le -\frac{1}{2} \frac{\rm d}{{\rm d}t} \int_{I} \phi_{t}^{2} \rho   (3\phi^{2}-1)    {\rm d}x + C \left( \| \sqrt{\rho}\phi_{t} \|_{L^{2}}^{2} +  \| \phi_{xt} \|_{L^{2}}^{2} \right)
\\
& \quad + Ct \| \sqrt{\rho}\phi_{t} \|_{L^{2}}^{4}  +  \frac{C}{t} \left( \| \sqrt{\rho}\phi_{t} \|_{L^{2}}^{2} +  \| \phi_{xt} \|_{L^{2}}^{2} \right).
\end{align*}
Substituting $H_{1}$--$H_{6}$ into \eqref{dt-6}, we arrive at
\begin{align}\label{dt-6-p}
& \frac{1}{2} \frac{\rm d}{{\rm d}t} \|  \phi_{xt} \|_{L^{2}}^{2} + \|  \mu_{xt} \|_{L^{2}}^{2}
\nonumber \\
& \le \frac{\rm d}{{\rm d}t} \mathcal{A}(t)  +  5\varepsilon \| \mu_{t} \|_{L^{\infty}}^{2} + \frac{1}{6} \| \mu_{xt} \|_{L^{2}}^{2} + C \| \sqrt{\rho} u_{t} \|_{L^{2}}^{2}  +  \frac{C}{t} \left( \| \sqrt{\rho}\phi_{t} \|_{L^{2}}^{2} +  \| \phi_{xt} \|_{L^{2}}^{2} \right)
\nonumber \\
& \quad + C \left(1+ \| \sqrt{\rho} u_{t} \|_{L^{2}}^{2} + t \| \sqrt{\rho}\phi_{t} \|_{L^{2}}^{2}  \right) \left( \| \sqrt{\rho}\phi_{t} \|_{L^{2}}^{2} +  \| \phi_{xt} \|_{L^{2}}^{2} \right)  +C\| \mu \|_{H^{1}}^{2} + C,
\end{align}
where
\begin{align}
\mathcal{A}(t) & := \int_{I} \phi_{xt} \rho u [ \mu- (\phi^{3}-\phi)  ]  {\rm d}x
\nonumber \\
& \quad +  \int_{I} \phi_{t} \rho u [ \mu_{x}- (3\phi^{2}-1)\phi_{x}  ]  {\rm d}x -\frac{1}{2} \int_{I} \phi_{t}^{2} \rho   (3\phi^{2}-1)    {\rm d}x.
\end{align}

{\bf Step 3. Estimate for $\| \mu_{t}\|_{L^{\infty}}$.}

In order to use Gronwall's inequality for \eqref{dt-6-p}, we need to estimate  $\| \mu_{t}\|_{L^{\infty}}$. Note that
\begin{align*}
\left | \int_{I} \rho \mu_{t} {\rm d}x \right | & =  \left | - \int_{I} \rho_{t} \mu {\rm d}x  + \int_{I} \rho_{t} (\phi^{3}-\phi) {\rm d}x + \int_{I} \rho (3\phi^{2}-1) \phi_{t}  {\rm d}x \right |
\\
& \le \| \rho_{t} \|_{L^{2}} \| \mu \|_{L^{2}} + \| \rho_{t} \|_{L^{2}} \| \phi^{3} -\phi \|_{L^{2}} + C \| 3\phi^{2} -1 \|_{L^{2}} \| \sqrt{\rho} \phi_{t} \|_{L^{2}}
\\
& \le C \| \sqrt{\rho} \phi_{t} \|_{L^{2}} + C \| \mu \|_{L^{2}} + C.
\end{align*}
Thus, similar to \eqref{mu00}, we have
\begin{align}\label{mu-t-00}
\| \mu_{t}\|_{L^{\infty}} \le   \| \mu_{xt}\|_{L^{2}}  + \left |  \int_{I} \rho \mu_{t} {\rm d}x  \right | \le  \| \mu_{xt}\|_{L^{2}} + C \| \sqrt{\rho} \phi_{t} \|_{L^{2}} + C \| \mu \|_{L^{2}} + C.
\end{align}
Putting \eqref{mu-t-00} into \eqref{dt-6-p},  and then multiplying the result by $t^{2}$, we obtain
\begin{align}\label{dt-6-pp}
& \frac{1}{2} \frac{\rm d}{{\rm d}t} \left(t^{2} \|  \phi_{xt} \|_{L^{2}}^{2} \right) + t^{2} \|  \mu_{xt} \|_{L^{2}}^{2}
\nonumber \\
& \le \frac{\rm d}{{\rm d}t} \left( t^{2} \mathcal{A}(t) \right) -  2t \mathcal{A}(t) +  5\varepsilon t^{2}\| \mu_{xt} \|_{L^{2}}^{2} + \frac{1}{6} t^{2}\| \mu_{xt} \|_{L^{2}}^{2}
\nonumber \\
& \quad  + C(T) \left(1+ \| \sqrt{\rho} u_{t} \|_{L^{2}}^{2} + \| \mu \|_{H^{1}}^{2} +  t\| \sqrt{\rho}\phi_{t} \|_{L^{2}}^{2} + t \|  \phi_{xt} \|_{L^{2}}^{2} \right)
\nonumber \\
& \quad + C \left(1+ \| \sqrt{\rho} u_{t} \|_{L^{2}}^{2} + t \| \sqrt{\rho}\phi_{t} \|_{L^{2}}^{2}  \right) \left( t^{2}\| \sqrt{\rho}\phi_{t} \|_{L^{2}}^{2} + t^{2}\| \phi_{xt} \|_{L^{2}}^{2} \right).
\end{align}

{\bf Step 4. Closure of the estimates.}

Multiplying \eqref{dt-5.5} by $2$, adding the result to \eqref{dt-6-pp}, substituting \eqref{mu-t-00} into \eqref{dt-6-p}, choosing $\varepsilon$ sufficiently small such that $5\varepsilon< 1/ 6$, we have
\begin{align}\label{gj}
& t^{2}\| \sqrt{\rho}\phi_{t}  \|_{L^{2}}^{2}  + \frac{1}{2}  t^{2}\|  \phi_{xt} \|_{L^{2}}^{2} +  \frac{1}{2} \int_{0}^{t} s^{2}\|  \mu_{xt} \|_{L^{2}}^{2} {\rm d}s
\nonumber \\
 \le&   t^{2} \mathcal{A}(t) -  2 \int_{0}^{t} s \mathcal{A}(s) {\rm d}s + C(T)
\nonumber \\
& + C \int_{0}^{t} \left(1+\| \phi_{xx} \|_{L^{2}}^{2} + \| \sqrt{\rho} u_{t} \|_{L^{2}}^{2} + s\| \sqrt{\rho}\phi_{t} \|_{L^{2}}^{2}  \right) \left( s^{2}\| \sqrt{\rho}\phi_{t} \|_{L^{2}}^{2} + s^{2} \| \phi_{xt} \|_{L^{2}}^{2} \right) (s) {\rm d}s.
\end{align}
Note that $-(3\phi^{2}-1) \le 1$. Hence, using \eqref{0}, \eqref{phi-b}, \eqref{u-x} and \eqref{mu-1}, one has
\begin{align}\label{a(t)}
\mathcal{A}(t) &\le \frac{1}{2} \| \sqrt{\rho}\phi_{t} \|_{L^{2}}^{2} + \int_{I} \phi_{xt} \rho u [ \mu- (\phi^{3}-\phi)  ]  {\rm d}x +  \int_{I} \phi_{t} \rho u [ \mu_{x}- (3\phi^{2}-1)\phi_{x}  ]  {\rm d}x
\nonumber \\
& \le \frac{3}{4} \| \sqrt{\rho}\phi_{t} \|_{L^{2}}^{2} +  \frac{1}{4} \|  \phi_{xt}  \|_{L^{2}}^{2} + C\| u \|_{L^{\infty}}^{2} \left( \| \mu \|_{H^{1}}^{2} + \| \phi^{3} - \phi \|_{L^{2}}^{2} + \| 3\phi^{2} - 1 \|_{L^{\infty}}^{2} \| \phi_{x} \|_{L^{2}}^{2} \right)
\nonumber \\
& \le \frac{3}{4} \| \sqrt{\rho}\phi_{t} \|_{L^{2}}^{2} +  \frac{1}{4} \|  \phi_{xt}  \|_{L^{2}}^{2} + C \| \mu \|_{H^{1}}^{2}+ C.
\end{align}
Substituting \eqref{a(t)} into \eqref{gj}, using Gronwall's inequality and \eqref{mu-t-00}, then utilizing \eqref{key}, as well as using \eqref{nsch-1d}$_{3,4}$,  we complete the proof of Lemma \ref{L-11}.
\end{proof}

\begin{lemma}\label{L-12}
Let $(\rho, u, \phi)$ be the smooth solutions to \eqref{nsch-1d}--\eqref{ic}, then it holds that
\begin{align}
\sup\limits_{0\le t \le T} \left( t \| \sqrt{\rho} u_{t} \|_{L^{2}}^{2} + t \| u_{xx} \|_{L^{2}}^{2} \right) + \int_{0}^{T}  (t \| u_{t} \|_{H^{1}}^{2} + t^{2} \| u_{xxx} \|_{L^{2}}^{2}) {\rm d}t  \le C(T).
\end{align}
\end{lemma}

\begin{proof}[\bf Proof]
Differentiating \eqref{nsch-1d}$_{2}$ with respect to $t$ and testing the result by $u_{t}$, then integrating over $I$ by parts, we have
\begin{align}\label{dt-7}
& \frac{1}{2} \frac{\rm d}{{\rm d}t} \| \sqrt{\rho} u_{t} \|_{L^{2}}^{2} + \| u_{xt} \|_{L^{2}}^{2}
\nonumber \\
 =&- 2 \int_{I} \rho u u_{t} u_{xt} {\rm d}x - \int_{I} \rho u u_{x}^{2} u_{t} {\rm d}x - \int_{I} \rho u^{2} u_{xx} u_{t} {\rm d}x - \int_{I} \rho u^{2} u_{x} u_{xt} {\rm d}x
\nonumber \\
& - \int_{I} \rho u_{x} u_{t}^{2} {\rm d}x - \gamma  \int_{I} \rho^{\gamma} u_{x} u_{xt} {\rm d}x - \gamma  \int_{I} \rho^{\gamma-1} \rho_{x} u u_{xt} {\rm d}x + \int_{I} \phi_{x} \phi_{xt} u_{xt} {\rm d}x
\nonumber \\
\le&  C \| u \|_{L^{\infty}} \| \sqrt{\rho} u_{t} \|_{L^{2}} \| u_{xt} \|_{L^{2}} + C\| u \|_{L^{\infty}} \| u_{x} \|_{L^{4}}^{2} \| \sqrt{\rho} u_{t} \|_{L^{2}}
\nonumber \\
&  + C \| u \|_{L^{\infty}}^{2} \left( \| u_{xx} \|_{L^{2}}  \| \sqrt{\rho} u_{t} \|_{L^{2}} + \| u_{x} \|_{L^{2}}  \| u_{xt} \|_{L^{2}} \right) + C \| u_{x} \|_{L^{\infty}}  \| \sqrt{\rho} u_{t} \|_{L^{2}}^{2}
\nonumber \\
&  + C  \| u_{x} \|_{L^{2}}  \| u_{xt} \|_{L^{2}} + C   \| u \|_{L^{\infty}}  \| \rho_{x} \|_{L^{2}} \| u_{xt} \|_{L^{2}}  + \int_{I} \phi_{x} \phi_{xt} u_{xt} {\rm d}x
\nonumber \\
\le &  \frac{1}{4} \| u_{xt} \|_{L^{2}}^{2} + C\left( \| \sqrt{\rho} u_{t} \|_{L^{2}}^{2} + 1 \right) \| \sqrt{\rho} u_{t} \|_{L^{2}}^{2} +  C \| u_{xx} \|_{L^{2}}^{2}   + C + \int_{I} \phi_{x} \phi_{xt} u_{xt} {\rm d}x.
\end{align}
It follows from \eqref{nsch-1d}$_{2}$ that
\begin{align}\label{u-222}
\| u_{xx} \|_{L^{2}}^{2}  \le C \| \sqrt{\rho} u_{t} \|_{L^{2}}^{2} + C\| \phi_{xx} \|_{H^{1}}^{2} + C.
\end{align}
We are going to estimate the last term on the right hand side  of \eqref{dt-7}. Using integration by parts formula yields
\begin{align}\label{u-kn}
\int_{I} \phi_{x} \phi_{xt} u_{xt} {\rm d}x = -\int_{I} \phi_{xx}\phi_{xt}u_{t}  {\rm d}x  -\int_{I} \phi_{x}\phi_{xxt}u_{t}  {\rm d}x =: J_{1} + J_{2}.
\end{align}
For $J_{1}$, it follows from \eqref{nsch-1d}$_{4}$ that
\begin{align}\label{u-kn-1}
J_{1} &  = \int_{I} \left( \mu - (\phi^{3} - \phi)  \right) \phi_{xt} \rho u_{t}  {\rm d}x
\nonumber \\
&\le  C \left( \| \mu \|_{L^{\infty}} + \| \phi^{3} - \phi \|_{L^{\infty}} \right) \| \phi_{xt} \|_{L^{2}} \| \sqrt{\rho} u_{t} \|_{L^{2}}
\nonumber \\
&\le C \left( \| \mu \|_{H^{1}}^{2} + 1 \right)  \| \sqrt{\rho} u_{t} \|_{L^{2}}^{2} + C \| \phi_{xt} \|_{L^{2}}^{2}.
\end{align}
For $J_{2}$, it also follows from \eqref{nsch-1d}$_{4}$ that
\begin{align}\label{u-kn-2}
J_{2} &  = \int_{I} \left( \rho_{t} \mu + \rho \mu_{t} - \rho_{t} (\phi^{3}-\phi) - \rho(3\phi^{2}-1)\phi_{t}  \right) \phi_{x} u_{t}  {\rm d}x
\nonumber \\
& \le \| \rho_{t} \|_{L^{2}} \| \phi_{x} \|_{L^{2}} \| u_{t} \|_{L^{\infty}} \left(  \| \mu \|_{L^{\infty}} + \| \phi^{3} - \phi\|_{L^{\infty}} \right)
\nonumber \\
& \quad + C \| \sqrt{\rho} u_{t} \|_{L^{2}} \| \phi_{x} \|_{L^{2}}  \left( \| \mu_{t} \|_{L^{\infty}} + \| 3\phi^{2} -1\|_{L^{\infty}} \| \phi_{t} \|_{L^{\infty}}  \right)
\nonumber \\
& \le C\| u_{xt} \|_{L^{2}} \left( \| \mu \|_{H^{1}} + 1 \right) + C \| \sqrt{\rho} u_{t} \|_{L^{2}}  \left( \| \mu_{xt} \|_{L^{2}} + \| \sqrt{\rho} \phi_{t} \|_{L^{2}} + C  +  \| \phi_{t} \|_{H^{1}}  \right)
\nonumber \\
& \le  \frac{1}{4}\| u_{xt} \|_{L^{2}}^{2} + C \left( \| \mu \|_{H^{1}}^{2} + 1 \right) + C t^{-1} \| \sqrt{\rho} u_{t} \|_{L^{2}}^{2}
\nonumber \\
& \quad  + C t \left( \| \mu_{xt} \|_{L^{2}}^{2} + \| \sqrt{\rho} \phi_{t} \|_{L^{2}}^{2} + C  +  \| \phi_{t} \|_{H^{1}}^{2}  \right).
\end{align}

Combining  \eqref{u-222}-\eqref{u-kn-2} with \eqref{dt-7}, then multiplying the result by $t$ and using \eqref{mu-1}, one obtains
\begin{align}\label{dt-77}
& \frac{1}{2} \frac{\rm d}{{\rm d}t} \left( t\| \sqrt{\rho} u_{t} \|_{L^{2}}^{2} \right) + t\| u_{xt} \|_{L^{2}}^{2}
\nonumber \\
& \le C\| \sqrt{\rho} u_{t} \|_{L^{2}}^{2}+C\left( \| \sqrt{\rho} u_{t} \|_{L^{2}}^{2} + \| \mu \|_{H^{1}}^{2} + 1 \right) t \| \sqrt{\rho} u_{t} \|_{L^{2}}^{2}
\nonumber \\
& \quad + Ct \| \mu \|_{H^{1}}^{2} +  Ct \| \phi_{xx} \|_{H^{1}}^{2}
+ C   t \| \phi_{t} \|_{H^{1}}^{2} + C t\| \sqrt{\rho} \phi_{t} \|_{L^{2}}^{2} + C t^{2} \| \mu_{xt} \|_{L^{2}}^{2} + C(T).
\end{align}
Thanks to \eqref{0}, \eqref{t-phi-2}, \eqref{e-6}, \eqref{u-x}, \eqref{mu-1} and \eqref{e-11}, using Gronwall's inequality to \eqref{dt-77}, and then using Poincar\'e's inequality, \eqref{u-222} and \eqref{nsch-1d}$_{2}$, we complete the proof of Lemma \ref{L-12}.
\end{proof}

Finally, as a consequence of Lemmas~\ref{L-0}--\ref{L-12}, we obtain the following
{\it a priori estimates}.
\begin{proposition}\label{p-estimate}
Assume that $(\rho, u, \phi)$ are smooth solutions to \eqref{nsch-1d}--\eqref{ic}, then there exists a generic constant $C$ depending on the initial data and $T$, such that
\begin{align}\label{zonge1}
\sup\limits_{0\le t \le T} & \left( \| (\rho, \phi)\|_{H^{1}}^{2} +  \| (\sqrt{\rho}u, \rho_{t}, u_{x} )\|_{L^{2}}^{2} \right)
\nonumber \\
& \qquad + \int_{0}^{T}  \left( \| \mu\|_{H^{1}}^{2} +  \| (u_{xx}, \phi_{xx},\sqrt{\rho}u_{t})\|_{L^{2}}^{2} \right) {\rm d}t \le C,
\end{align}
and
\begin{align}\label{zonge2}
\sup\limits_{0\le t \le T} & \left( \| (\sqrt{t}\phi_{xx}, \sqrt{t}\phi_{xxx}, \sqrt{t} \sqrt{\rho} u_{t},  \sqrt{t}u_{xx}, t \sqrt{\rho}\phi_{t}, t\mu_{xx} )\|_{L^{2}}^{2} + \| (\sqrt{t} \mu, t \phi_{t})\|_{H^{1}}^{2} \right)
\nonumber \\
& \qquad + \int_{0}^{T}  \left(    \| (\sqrt{t} \phi_{t}, t \mu_{t}, t\phi_{xt}, \sqrt{t} u_{t})\|_{H^{1}}^{2} + \| (t^{\frac{1}{4}}\sqrt{\rho}\phi_{t}, tu_{xxx})\|_{L^{2}}^{2} \right) {\rm d}t \le C.
\end{align}
\end{proposition}

\section{Existence And Uniqueness}\label{sec3}
In this section, we will prove Theorem \ref{TH}. In order to prove Theorem \ref{TH}, we start with the following global well-posedness for the case of positive initial densities, which can be proved in the same way as in \cite[Theorem 1]{D-L}.
\begin{proposition}\label{p-gw}
Let $\rho_{0} \in C^{3,\alpha}(I)$ satisfies $0<\underline{\rho} \le \rho_{0} \le \bar{\rho}$ for some constants $\alpha \in (0,1), \underline{\rho}$ and $\bar{\rho}$, $u_{0} \in C^{3,\alpha}(I)$ with $u_{0}(0) = u_{0}(1) = 0$ and $\phi_{0} \in C^{4,\alpha}(I)$. Then the problem \eqref{nsch-1d}--\eqref{ic} admits a unique classical solution $(\rho, u, \phi)$ satisfying that, for any $T>0$, there exists a constant $C=C(\underline{\rho}, \bar{\rho}, T)>0$ such that
\begin{alignat*}{2}
&(\rho_{xxx}, \rho_{xxt} ) \in C^{\frac{\alpha}{2},\frac{\alpha}{4}}(\overline{Q}_{T}), &&\qquad  0 < C^{-1} \le \rho  \le C \quad on~Q_{T},
\\
& ~ u_{x}  \in C^{2+\frac{\alpha}{2},1+\frac{\alpha}{4}}(\overline{Q}_{T}),&& \qquad \phi \in C^{4+\frac{\alpha}{2},1+\frac{\alpha}{4}}(\overline{Q}_{T}).
\end{alignat*}
\end{proposition}

Based on the {\it a priori estimates} derived in the Section \ref{sec2}, which is independent of the lower bound of $\rho_{0}$, we can extend the global well-posedness obtained in Proposition \ref{p-gw} to the case where the initial density is allowed to vacuum.

We are now ready to give the proof of Theorem \ref{TH}.

\begin{proof}[\bf Proof  of Theorem \ref{TH}]
{\bf (i) Existence.} The proof of existence will be completed by several steps.

{\bf Step 1. Construction of the initial data.}
Recall the initial data satisfy
\[
0\le \rho_0\in H^1(I), \qquad
u_0\in H_0^1(I), \qquad \phi_0\in H^1(I).
\]
Choose nonnegative smooth functions $r_n\in C^\infty(\bar I)$ such that
$r_n\to\rho_0$ in $H^1(I)$, and then set
$$
\rho_{0n}
=
\frac{r_n+n^{-1}}{\int_I(r_n+n^{-1})\,{\rm d}x}.
$$
By the standard density argument, there exist sequences
$$
\rho_{0n}\in C^\infty(\bar I),\qquad
u_{0n}\in C_c^\infty(I),\qquad
\phi_{0n}\in C^\infty(\bar I),
$$
such that
$$
\rho_{0n}>0,\qquad \int_I\rho_{0n}\,{\rm d}x=1,
$$
and
$$
\rho_{0n}\to\rho_0,\qquad
u_{0n}\to u_0,\qquad
\phi_{0n}\to\phi_0
\quad\text{in }H^1(I).
$$
In addition, there exists a constant $C>0$, such that
\[
\|\rho_{0n}\|_{C^{3,\alpha}}
+\|u_{0n}\|_{C^{3,\alpha}}
+\|\phi_{0n}\|_{C^{4,\alpha}}
\le C.
\]

{\bf Step 2. Approximate solutions and convergence.}

By Proposition~\ref{p-gw}, for each $n\in\mathbb{N}$,
	there exists a unique global classical solution
	$(\rho_n,u_n,\phi_n)$ to problem \eqref{nsch-1d}--\eqref{ic}
	corresponding to the initial data
	$(\rho_{0n},u_{0n},\phi_{0n})$.  From Proposition \ref{p-estimate}, there are two positive constants $T$ and $C$ independent of $n$, such that the solution $(\rho_{n}, u_{n}, \phi_{n})$ satisfying the following a priori estimates
\begin{align}\label{n-zonge1}
\sup\limits_{0\le t \le T} & \left( \| (\rho_{n}, \phi_{n})\|_{H^{1}}^{2} +  \| (\sqrt{\rho_{n}}u_{n},\rho_{nt},   u_{nx} )\|_{L^{2}}^{2} \right)
\nonumber \\
& \quad + \int_{0}^{T}  \left( \| \mu_{n}\|_{H^{1}}^{2} +  \| (u_{nxx}, \phi_{nxx},\sqrt{\rho_{n}}u_{nt})\|_{L^{2}}^{2} \right) {\rm d}t \le C,
\end{align}
and
\begin{align}\label{n-zonge2}
\sup\limits_{0\le t \le T} & \left( \| (\sqrt{t}\phi_{nxx}, \sqrt{t}\phi_{nxxx}, \sqrt{t} \sqrt{\rho_{n}} u_{nt},  t \sqrt{\rho_{n}}\phi_{nt}, t \phi_{nxt}, t\mu_{nxx} )\|_{L^{2}}^{2} + \| \sqrt{t} \mu_{n}\|_{H^{1}}^{2} \right)
\nonumber \\
& \qquad  + \int_{0}^{T}  \left(    \| (\sqrt{t} \phi_{nt}, t \mu_{nt}, t\phi_{xt}, \sqrt{t} u_{nt})\|_{H^{1}}^{2} + \| (t^{\frac{1}{4}}\sqrt{\rho_{n}}\phi_{nt}, tu_{nxxx})\|_{L^{2}}^{2} \right) {\rm d}t \le C.
\end{align}
Thus, by the Banach-Alaoglu theorem and using the Cantor's diagonal arguments, there is a subsequence, still denoted by $(\rho_{n}, u_{n}, \phi_{n}, \mu_{n})$, such that, as $n\to \infty$, we have for any $\delta \in (0, T)$
\begin{alignat}{3}
& \rho_{n}  \stackrel{*}{\rightharpoonup} \rho &\quad& {\rm in}~L^{\infty}(0,T;H^{1}), &\qquad& \rho_{nt}  \stackrel{*}{\rightharpoonup} \rho_{t} \quad {\rm in}~L^{\infty}(0,T;L^{2}),
\label{w-rho}\\
& u_{n}  \stackrel{*}{\rightharpoonup} u &\quad& {\rm in}~L^{\infty}(0,T;H_{0}^{1}), &\qquad& u_{nt}  \stackrel{}{\rightharpoonup} u_{t} \quad {\rm in}~L^{2}(\delta,T; H^{2}),
\label{w-u}\\
& \phi_{n}  \stackrel{*}{\rightharpoonup} \phi &\quad& {\rm in}~L^{\infty}(0,T;H^{1}), &\qquad& \phi_{nt}  \stackrel{}{\rightharpoonup} \phi_{t} \quad {\rm in}~L^{2}(\delta,T; H^{1}),
\label{w-phi} \\
& u_{n}  \stackrel{}{\rightharpoonup} u &\quad& {\rm in}~L^{2}(0,T;H^{2}), &\qquad& \phi_{n}  \stackrel{}{\rightharpoonup} \phi \quad \,\,\, \, \,{\rm in}~L^{2}(0,T; H^{2}),
\label{w-pphi}
\\
& \mu_{n}  \stackrel{}{\rightharpoonup} \mu &\quad& {\rm in}~L^{2}(0,T;H^{1}), &\qquad& t\mu_{tn}  \stackrel{}{\rightharpoonup} t\mu_{t}  \,\,\,{\rm in}~L^{2}(0,T; H^{1}),
\label{w-mu}
\end{alignat}
and $(\rho, u, \phi)$ satisfying
\begin{alignat}{2}
& \rho \in L^{\infty}(0,T;H^{1}), &\qquad& \rho_{t} \in L^{\infty}(0,T;L^{2}),
\label{r-rho}\\
& u \in L^{\infty}(0,T;H_{0}^{1}) \cap L^{2}(0,T; H^{2}), &\qquad& \sqrt{\rho} u_{t} \in L^{2}(0,T;L^{2}),
\\
& t u \in L^{\infty}(0,T;H^{2}) \cap L^{2}(0,T; H^{3}), &\qquad& \sqrt{t} u_{t} \in  L^{\infty}(0,T; H^{1}),
\\
& \phi \in L^{\infty}(0,T; H^{1}) \cap L^{2}(0,T; H^{2}), &\qquad& \sqrt{t} \phi_{xx} \in L^{\infty}(0,T;H^{1}),
\\
& \sqrt{t} \phi_{t} \in L^{\infty}(0,T;H^{1}),  &\qquad&   t\phi_{xxt} \in  L^{2}(0,T; L^{2})
\\
&\sqrt{t} \mu \in  L^{\infty}(0,T; H^{1}),  &\qquad&  \mu \in  L^{2}(0,T; H^{1}),
\\
& t\mu_{xx} \in L^{\infty}(0,T; L^{2}),   &\qquad& t \mu_{t} \in  L^{2}(0,T; H^{1}). \label{r-mu}
\end{alignat}
In addition, since $H^{1}(I) \hookrightarrow \hookrightarrow C(\bar{I})$ and $H^{2}(I)\hookrightarrow\hookrightarrow H^{1}(I)\hookrightarrow\hookrightarrow L^{2}(I) $, it follows from Aubin-Lions lemma and \eqref{w-rho}--\eqref{w-mu} that
\begin{alignat}{2}
&\rho_{n} \to \rho,  &\qquad \quad& {\rm in}~C([0,T];C),
\label{s-rho}\\
&u_{n} \to u, &\qquad \quad& {\rm in}~C([\delta,T];L^{2})\cap L^{2}(\delta, T; H^{1}_{0}),
\label{s-u}\\
&\phi_{n} \to \phi, &\qquad \quad& {\rm in}~C([\delta,T];L^{2})\cap L^{2}(\delta, T; H^{1}). \label{s-phi}
\end{alignat}
Due to the convergence \eqref{w-rho}--\eqref{w-mu} and \eqref{s-rho}--\eqref{s-phi}, we have the convergence of the nonlinear terms as follows
\begin{alignat}{2}
&(\rho_{n}u_{n}, \sqrt{\rho_{n}}u_{n}, \rho_{n}\phi_{n}, \sqrt{\rho_{n}}\phi_{n} ) \to  (\rho u, \sqrt{\rho}u, \rho \phi, \sqrt{\rho}\phi),  &\quad& {\rm in}~C([\delta,T];L^{2}),
\label{n-rho}\\
&(\rho_{n}u_{nt}, \sqrt{\rho_{n}}u_{nt}, \rho_{n}\phi_{nt}, \sqrt{\rho_{n}}\phi_{nt} ) \rightharpoonup (\rho u_{t}, \sqrt{\rho}u_{t}, \rho \phi_{t}, \sqrt{\rho}\phi_{t}),  &\quad& {\rm in}~L^{2}(\delta,T;L^{2}),
\label{n-rhot}\\
&( \rho_{n} u_{n} u_{nx}, \rho_{n} u_{n} \phi_{nx}) \to (\rho u u_{x}, \rho u \phi_{x}), &\quad & {\rm in}~L^{1}((\delta,T)\times I),
\label{n-phi} \\
&\phi_{nx} \phi_{nxx} \rightharpoonup \phi_{x} \phi_{xx}, &\quad & {\rm in}~L^{1}((\delta,T)\times I),
\label{n-nn} \\
&\rho_{n} u_{nx} \to \rho u_{x}, &\quad & {\rm in}~L^{2}((\delta,T)\times I),
\label{n-midu1} \\
&\rho_{nx} u_{n} \to \rho_{x} u, &\quad & {\rm in}~C([\delta,T]; L^{1}), \label{n-midu2}
\end{alignat}
for any $\delta \in (0,T)$.

{\bf Step 3. The existence.}

With the convergence \eqref{w-rho}--\eqref{w-mu} and \eqref{s-rho}--\eqref{n-midu2} at hand, we take the limit as $n \to \infty$ to the
equations of $(\rho_{n}, u_{n}, \phi_{n})$ to show that $(\rho, u, \phi)$ satisfies equations \eqref{nsch-1d} in the sense of distribution. Thanks to the regularities \eqref{r-rho}--\eqref{r-mu}, one can further show that $(\rho, u, \phi)$ satisfies \eqref{nsch-1d}, a.e. in $I\times (0, T)$. The initial condition $\rho|_{t=0}=\rho_{0}$ is guaranteed by \eqref{s-rho} by recalling that  $\rho_{n} |_{t=0} = \eta_{1/n} * \rho_0 + \frac{1}{n}$.

It remains to prove $( \rho u,  \rho \phi)|_{t=0} = ( \rho_{0} u_{0},  \rho_{0}\phi_{0})$. Note that $  \rho u \in C((0,T];L^{2})$ and $ \rho \phi \in C((0,T];L^{2})$, we only need to show that
\begin{align}
 \rho u \to  \rho_{0} u_{0}, &\quad {\rm in}~ L^{2}, \quad {\rm as} ~t\to 0, \label{c-initial1}
\\
 \rho \phi \to   \rho_{0}  \phi_{0}, &\quad {\rm in}~ L^{2}, \quad {\rm as} ~t\to 0. \label{c-initial2}
\end{align}
Using \eqref{n-zonge1} and Holder inequality yields
\begin{align}
\int_{0}^{T} \| ( \rho_{n} u_{n})_{t} \|_{L^{2}}^{2} {\rm d}t & \le 2 \int_{0}^{T} \left( \|  \rho_{nt} u_{n} \|_{L^{2}}^{2}  + \|  \rho_{n} u_{nt}\|_{L^{2}}^{2} \right) {\rm d}t
\nonumber\\
& \le C \int_{0}^{T} \left( \| \rho_{nt} \|_{L^{2}}^{2}  \| u_{n} \|_{L^{\infty}}^{2} +  \| \sqrt{\rho_{n}} u_{nt}\|_{L^{2}}^{2} \right) {\rm d}t \le C,
\end{align}
for sufficiently large $n$. Consequently, it follows from the Newton-Leibnitz formula, the Minkowski inequality that
\begin{align}\label{i-pu-2}
\| \rho u(\cdot, t) - \rho_{0} u_{0} \|_{L^{2}} \le & \| \rho u - \rho_{n} u_{n} \|_{L^{2}} (t) + \|  \rho_{n} u_{n} (\cdot, t) - \rho_{0n} u_{0n} \|_{L^{2}}
\nonumber \\
& + \| \rho_{0n} u_{0n} - \rho_{0n} u_{0} \|_{L^{2}} + \| \rho_{0n} u_{0} - \rho_{0} u_{0}\|_{L^{2}}
\nonumber \\
\le &  \| \rho u - \rho_{n} u_{n} \|_{L^{2}} (t) + \int_{0}^{t} \|( \rho_{n} u_{n})_{t}\|_{L^{2}}  {\rm d}s
\nonumber \\
&  + \| \rho_{0n}\|_{L^{\infty}} \| u_{0n} - u_{0} \|_{L^{2}} + \frac{C}{n} \| u_{0} \|_{L^{2}}
\nonumber \\
\le &  \| \rho u - \rho_{n} u_{n} \|_{L^{2}} (t) + C\sqrt{t} + C\| u_{0n} - u_{0} \|_{L^{2}} + \frac{C}{n} \| u_{0} \|_{L^{2}},
\end{align}
for sufficiently large $n$. Recalling \eqref{n-rho} and the convergence $u_{0n} \to u_{0}$ in $H^{1}$ as
$n \to \infty$, we may pass to the limit in \eqref{i-pu-2} to conclude that
$$
\| \rho u(\cdot,t) - \rho_{0}u_{0} \|_{L^{2}}
\le C \sqrt{t},
$$
which proves \eqref{c-initial1}.

Using \eqref{n-zonge2} and Holder inequality yields
\begin{align}
\int_{0}^{T} t^{\frac{1}{2}}\| (\rho_{n} \phi_{n})_{t} \|_{L^{2}}^{2} {\rm d}t & \le 2 \int_{0}^{T} t^{\frac{1}{2}}\left( \| \rho_{nt} \phi_{n} \|_{L^{2}}^{2}  + \| \rho_{n} \phi_{nt}\|_{L^{2}}^{2} \right) {\rm d}t
\nonumber\\
& \le C \int_{0}^{T} t^{\frac{1}{2}}\left( \| \rho_{nt} \|_{L^{2}}^{2}  \| \phi_{n} \|_{L^{\infty}}^{2} + \| \sqrt{\rho_{n}}\phi_{nt}\|_{L^{2}}^{2} \right) {\rm d}t \le C,
\end{align}
for large $n$. Similar to \eqref{i-pu-2}, we have
\begin{align}\label{i-pphi-2}
\| \rho \phi(\cdot, t) - \rho_{0} \phi_{0} \|_{L^{2}}
\le &  \| \rho \phi- \rho_{n} \phi_{n} \|_{L^{2}} (t) + \| \rho_{0n}\|_{L^{\infty}} \| \phi_{0n} - \phi_{0} \|_{L^{2}} + \frac{C}{n} \| \phi_{0} \|_{L^{2}}
\nonumber \\
&  + \left( \int_{0}^{t} s^{-\frac{1}{2}}  {\rm d}s \right)^{\frac{1}{2}}\left( \int_{0}^{t} \left( s^{\frac{1}{4}}\|( \rho_{n} \phi_{n})_{t}\|_{L^{2}} \right)^{2} {\rm d}s \right)^{\frac{1}{2}}
\nonumber \\
\le &  \| \rho \phi- \rho_{n} \phi_{n} \|_{L^{2}} (t) + Ct^{\frac{1}{4}} + C\| \phi_{0n} - \phi_{0} \|_{L^{2}} + \frac{C}{n} \| \phi_{0} \|_{L^{2}},
\end{align}
for large $n$. Noting \eqref{n-rho} and $\phi_{0n} \to \phi_{0}$ in $L^{2}$ as $n \to \infty$, we obtain by taking $n \to \infty$ that $\| \rho\phi(\cdot, t) - \rho_{0} \phi_{0} \|_{L^{2}} \le Ct^{\frac{1}{4}}$, which indicates \eqref{c-initial2}.
\, \\[-1em]
{\bf (ii) Uniqueness.} Let $(\rho_{i}, u_{i}, \phi_{i} ) (i=1,2)$ be two solutions to the problem \eqref{nsch-1d} obtained above. Denote $\rho = \rho_{1} -\rho_{2}, u = u_{1} - u_{2}, \phi = \phi_{1} - \phi_{2}$ and $\mu = \mu_{1} - \mu_{2}$. Then $(\rho, u, \phi)$ satisfies the following equations
\begin{align}\label{nsch-w}
\begin{cases}
\rho_{t} + (\rho u_{1})_{x}  +   (\rho_{2} u)_{x}  = 0,
\\
\rho_{1} u_{t} - u_{xx}  =  - \rho u_{2t} - \rho_{1} u_{1} u_{x} -  \rho_{1} u u_{2x}  - \rho u_{2} u_{2x}
\\
\hspace{6em} - (\rho_{1}^{\gamma} - \rho_{2}^{\gamma})_{x} - \phi_{1x}\phi_{xx} - \phi_{x}\phi_{2xx} ,
\\
\rho_{1}\phi_{t} + \rho_{1} u_{1} \phi_{x} = \mu_{xx} - \rho \phi_{2t} - \rho u_{2} \phi_{2x} - \rho_{1} u \phi_{2x},
\\
\rho_{1}\mu = -\phi_{xx} + \rho_{1}(\phi_{1}^{2}+ \phi_{1} \phi_{2} + \phi_{2}^{2} -1) \phi + \rho(\phi_{2}^{3}-\phi_{2}) - \rho \mu_{2},
\end{cases}
\end{align}
subject to the boundary value condition
\begin{align}\label{bc-w}
(u, \phi_{x}, \mu_{x}) \Big|_{x=0,1} = (0, 0, 0).
\end{align}
It follows from the Holder and Poincare inequalities that
\begin{align}\label{chuzhi00}
\| \sqrt{\rho_{1}} u \|_{L^{2}}^{2}(t) & = \int_{I} \rho_{1}|u|^{2}(x,t) {\rm d}x \le \| \rho_{1} u \|_{L^{2}} (t) \|  u \|_{L^{2}} (t)
\nonumber \\
& \le C \big( \| \rho_{1} u_{1} - \rho_{2}u_{2} \|_{L^{2}} (t)  + \| (\rho_{1} - \rho_{2})u_{2} \|_{L^{2}} (t) \big)
\nonumber \\
& \le C \big( \| \rho_{1} u_{1} - \rho_{2}u_{2} \|_{L^{2}} (t)  + \| \rho_{1} - \rho_{2} \|_{L^{2}} (t) \| \nabla u_{2} \|_{L^{2}} (t) \big)
\nonumber \\
& \le C \big( \| \rho_{1} u_{1} - \rho_{2}u_{2} \|_{L^{2}} (t)  + \| \rho_{1} - \rho_{2} \|_{L^{2}} (t)  \big),
\end{align}
for every $t>0$. Note that $\rho_{1}(x,0) = \rho_{2}(x,0) $ and $(\rho_{1} u_{1}) (x,0) = (\rho_{2}u_{2}) (x,0) $. Therefore, let $t \to 0 $, we deduce from \eqref{chuzhi00} that
\begin{align}\label{u000}
\lim\limits_{t\to 0} \| \sqrt{\rho_{1}} u \|_{L^{2}}^{2}(t) =0.
\end{align}
Similar to \eqref{chuzhi00}, due to $ \nabla \phi_{2} \in L^{\infty}(0,T; L^{2})$, we also have
\begin{align}\label{phi000}
\lim\limits_{t\to 0} \| \sqrt{\rho_{1}} \phi \|_{L^{2}}^{2}(t) =0.
\end{align}

{\bf Step 1. Energy inequalities of $\rho$ and $u$. }

Testing \eqref{nsch-w}$_{1}$ by $\rho$ and integrating over $I$ by parts, one has
\begin{align}\label{wy-00}
 \frac{1}{2} \frac{\rm d}{{\rm d}t} \| \rho \|_{L^{2}}^{2} & = \int_{I} \rho u_{1} \rho_{x}  {\rm d}x - \int_{I}  \left( \rho_{2x} u + \rho_{2} u_{x} \right) \rho  {\rm d}x
 \nonumber \\
 & =  - \frac{1}{2} \int_{I} \rho^{2} u_{1x}  {\rm d}x - \int_{I}  \left( \rho_{2x} u + \rho_{2} u_{x} \right) \rho  {\rm d}x
\nonumber \\
 & \le  \frac{1}{2} \| u_{1x} \|_{L^{\infty}} \| \rho \|_{L^{2}}^{2} + \| u \|_{L^{\infty}} \| \rho_{2x} \|_{L^{2}}  \| \rho \|_{L^{2}}  +  \| \rho_{2} \|_{L^{\infty}} \| u_{x} \|_{L^{2}}  \| \rho \|_{L^{2}}
\nonumber \\
 & \le C \left( \| u_{1x} \|_{H^{1}} + 1 \right) \| \rho \|_{L^{2}}^{2} + C \| u_{x} \|_{L^{2}}^{2}.
\end{align}

Testing \eqref{nsch-w}$_{2}$ by $u$ and integrating over $I$ by parts, one deduces
\begin{align}\label{wy-11}
 & \frac{1}{2} \frac{\rm d}{{\rm d}t} \| \sqrt{\rho_{1}} u \|_{L^{2}}^{2} +  \| u_{x}\|_{L^{2}}^{2}
 \nonumber \\
 = & - \int_{I}  \rho u u_{2t} {\rm d}x - \int_{I}  \rho u u_{2} u_{2x} {\rm d}x  - \int_{I}  \rho_{1} u^{2} u_{2x} {\rm d}x + \int_{I}  \left( \rho_{1}^{\gamma} - \rho_{2}^{\gamma} \right)  u_{x} {\rm d}x
 \nonumber \\
 &  + \int_{I}  \phi_{1x} \phi_{x} u_{x} {\rm d}x + \int_{I}  \phi_{1xx} \phi_{x} u {\rm d}x - \int_{I}  \phi_{2xx} \phi_{x} u {\rm d}x
 \nonumber \\
  \le  &\| u\|_{L^{\infty}}  \| \rho \|_{L^{2}} \| u_{2t} \|_{L^{2}} + \| u\|_{L^{\infty}} \| u_{2}\|_{L^{\infty}} \| \rho \|_{L^{2}} \| u_{2x} \|_{L^{2}}
\nonumber \\
 &   + C\| u_{2x}\|_{L^{2}} \| \sqrt{\rho_{1}} u \|_{L^{2}} \|  u \|_{L^{\infty}} + C  \| \rho \|_{L^{2}} \| u_{x} \|_{L^{2}}
 \nonumber \\
 &   + \| \phi_{1x}\|_{L^{\infty}} \| u_{x}\|_{L^{2}} \| \phi_{x} \|_{L^{2}} + \| u\|_{L^{\infty}} \| \phi_{x}\|_{L^{2}} \left( \| \phi_{1xx}\|_{L^{2}}  + \| \phi_{2xx}\|_{L^{2}} \right)
 \nonumber \\
   \le & \frac{1}{2} \| u_{x}\|_{L^{2}}^{2} + C \left(\| u_{2t} \|_{L^{2}}^{2} +1 \right) \| \rho \|_{L^{2}}^{2} + C \| \sqrt{\rho_{1}} u \|_{L^{2}}^{2}
 \nonumber \\
 &  + C \left( \| \phi_{1xx}\|_{L^{2}}^{2} + \| \phi_{2xx}\|_{L^{2}}^{2} \right) \| \phi_{x}\|_{L^{2}}^{2}.
\end{align}

{\bf Step 2. Growth estimates.}

We proceed to consider the growth estimates of $\rho$, $\sqrt{\rho_{1}\phi}$ and $\phi_{x}$. Testing \eqref{nsch-w}$_{1}$ with $\rho$, one obtains that
\begin{align*}
 \frac{\rm d}{{\rm d}t} \| \rho \|_{L^{2}}^{2} & \le  C\| u_{1x} \|_{H^{1}} \| \rho \|_{L^{2}}^{2} + C \| u_{x} \|_{L^{2}} \| \rho_{2x} \|_{L^{2}}  \| \rho \|_{L^{2}}  +  C \| u_{x} \|_{L^{2}}  \| \rho \|_{L^{2}},
\end{align*}
which implies (if $\|\rho\|_{L^2}\neq0$)
\begin{align}\label{zz-rho}
 \frac{\rm d}{{\rm d}t} \| \rho \|_{L^{2}} & \le  C\| u_{1x} \|_{H^{1}} \| \rho \|_{L^{2}} + C \| u_{x} \|_{L^{2}} \| \rho_{2x} \|_{L^{2}}   +  C \| u_{x} \|_{L^{2}}.
\end{align}
Thanks to $ u_{x}, \rho_{2x} \in L^{\infty}(0, T; L^{2}) $ and $ u_{1x} \in L^{2}(0,T; H^{1})$, applying the Gronwall's inequality to \eqref{zz-rho}, we conclude that
\begin{align}\label{fzz-rho}
\| \rho \|_{L^{2}}  \le  Ce^{C\int_{0}^{t} \| u_{1x} \|_{H^{1}} {\rm d}s }  \int_{0}^{t}  \left( \| u_{x} \|_{L^{2}} \| \rho_{2x} \|_{L^{2}}    +   \| u_{x} \|_{L^{2}} \right) {\rm d}s \le Ct.
\end{align}

We now consider the growth estimates of $\sqrt{\rho_{1}}\phi$. Testing \eqref{nsch-w}$_{3}$ with $\phi$, noting $ \rho_{1t} = - (\rho_{1}u_{1})_{x}$, $\rho_{1t}\in L^{\infty}(0,T;L^{2})$ and $(\phi_{i}, u_{i}, \phi, u)\in L^{\infty}(0,T;H^{1}), i=1, 2$, we have
\begin{align}\label{zz-phi}
&  \frac{1}{2}\frac{\rm d}{{\rm d}t} \| \sqrt{\rho_{1}} \phi\|_{L^{2}}^{2} + \| \sqrt{\rho_{1}}  \mu \|_{L^{2}}^{2}
\nonumber \\
& =   - \int_{I} \rho\phi_{2t} \phi {\rm d}x - \int_{I} \rho u_{2}\phi_{2x} \phi {\rm d}x - \int_{I} \rho_{1} u \phi_{2x} \phi {\rm d}x
\nonumber \\
&  \quad  + \int_{I} \rho_{1}(\phi_{1}^{2}+ \phi_{1} \phi_{2} + \phi_{2}^{2}-1) \phi \mu  {\rm d}x
 - \int_{I} \rho [\mu_{2} - (\phi_{2}^{3}-\phi_{2}) ] \mu {\rm d}x
\nonumber \\
& \le  C \| \rho \|_{L^{2}} \| \phi_{2t} \|_{L^{2}} \| \phi \|_{H^1}
+ C \| \rho \|_{L^{2}} \| u_{2x} \|_{L^2}  \| \phi_{2x} \|_{L^{2}} \| \phi \|_{H^1}+ C \| u_{x} \|_{L^2}  \| \phi_{2x} \|_{L^{2}} \| \sqrt{\rho_{1}} \phi \|_{L^2}
\nonumber \\
&  \quad  + C \| \sqrt{\rho_{1}}\mu \|_{L^{2}} \| \phi_{1}^{2}+ \phi_{1} \phi_{2} + \phi_{2}^{2}-1 \|_{L^\infty} \| \sqrt{\rho_{1}} \phi \|_{L^2}
\nonumber \\
&  \quad  + \| \rho \|_{L^{2}} \| \mu \|_{L^2} ( \| \mu_{2} \|_{L^{\infty}} + \| \phi_{2}^{3} - \phi_{2} \|_{L^\infty})
\nonumber \\
& \le \frac{1}{2} \| \sqrt{\rho_{1}}  \mu \|_{L^{2}}^{2}  + C\| \rho \|_{L^{2}} (\| \phi_{2t} \|_{L^{2}} + 1 ) +C \| \rho \|_{L^{2}} \| \mu \|_{L^2} ( \| \mu_{2} \|_{H^{1}} +1)+C.
\end{align}
Due to $ (t\phi_{2t},  \sqrt{t} \mu_{2}) \in L^{\infty}(0,T;L^{2})$ and \eqref{fzz-rho}, integrating \eqref{zz-phi} over $(0,t)$, we deduce from \eqref{phi000} that
\begin{align}\label{zz-rho-phi}
& \| \sqrt{\rho_{1}} \phi\|_{L^{2}}^{2} (t) + \int_{0}^{t} \| \sqrt{\rho_{1}}  \mu \|_{L^{2}}^{2} {\rm d}s
\nonumber \\
& \le C\left( \int_{0}^{t} \| \rho \|_{L^{2}} (\| \phi_{2t} \|_{L^{2}} + 1 ) {\rm d}s+ \int_{0}^{t} \| \rho \|_{L^{2}} \| \mu \|_{L^2} ( \| \mu_{2} \|_{H^{1}} +1)  {\rm d}s +t\right)
\nonumber \\
& \le C \int_{0}^{t} s (\| \phi_{2t} \|_{L^{2}} + 1 ) {\rm d}s+ C\int_{0}^{t} s\| \mu \|_{L^2} ( \| \mu_{2} \|_{H^{1}} +1)  {\rm d}s + Ct
 \le Ct.
\end{align}
Finally, we consider the growth estimates of $\phi_{x}$. Testing \eqref{nsch-w}$_{4}$ by $\mu$, integrating the result over $I$ by parts, and using $\phi_{1},\phi_{2} \in L^{\infty}(0,T;H^{1})$, \eqref{fzz-rho} and \eqref{zz-rho-phi}, one has
\begin{align}
 \| \sqrt{\rho_{1}}  \mu \|_{L^{2}}^{2} & = \int_{I} \mu_{x} \phi_{x} {\rm d}x + \int_{I} \rho_{1}(\phi_{1}^{2}+ \phi_{1} \phi_{2} + \phi_{2}^{2}-1) \phi \mu  {\rm d}x - \int_{I} \rho [\mu_{2} - (\phi_{2}^{3}-\phi_{2}) ] \mu {\rm d}x
 \nonumber \\
& \le \frac{1}{2} \| \sqrt{\rho_{1}}  \mu \|_{L^{2}}^{2}  + \| \phi_{x}\|_{L^{2}} \| \mu_{x}\|_{L^{2}} + \| \phi_{1}^{2}+ \phi_{1} \phi_{2} + \phi_{2}^{2}-1 \|_{L^\infty}^{2} \|  \sqrt{\rho_{1}} \phi \|_{L^2}^{2}
 \nonumber \\
& \quad  + \| \rho \|_{L^{2}} \| \mu \|_{L^2} ( \| \mu_{2} \|_{H^{1}} +1)
 \nonumber \\
& \le \frac{1}{2} \| \sqrt{\rho_{1}}  \mu \|_{L^{2}}^{2}  +  C\| \mu_{x}\|_{L^{2}} + Ct
   + Ct \| \mu \|_{L^2} ( \| \mu_{2} \|_{H^{1}} +1)
\nonumber \\
& \le \frac{1}{2} \| \sqrt{\rho_{1}}  \mu \|_{L^{2}}^{2}  + C\| \mu_{x}\|_{L^{2}} + C,
\end{align}
which implies
\begin{align} \label{zz-mu}
 \| \sqrt{\rho_{1}}  \mu \|_{L^{2}}^{2} \le  C\| \mu_{x}\|_{L^{2}} + C.
\end{align}
Testing \eqref{nsch-w}$_{4}$ by $\phi$ , integrating the result over $I$ by parts, and using \eqref{zz-rho-phi} and \eqref{zz-mu}, one gets
\begin{align}\label{zz-phix}
 \| \phi_{x} \|_{L^{2}}^{2}(t)  & = \int_{I} \rho_{1}\mu \phi {\rm d}x - \int_{I} \rho_{1}(\phi_{1}^{2}+ \phi_{1} \phi_{2} + \phi_{2}^{2}-1) \phi^{2}  {\rm d}x + \int_{I} \rho [\mu_{2} - (\phi_{2}^{3}-\phi_{2}) ] \phi {\rm d}x
 \nonumber \\
& \le \| \sqrt{\rho_{1}}  \phi \|_{L^{2}} \| \sqrt{\rho_{1}}  \mu \|_{L^{2}} + C\| \sqrt{\rho_{1}}  \phi \|_{L^{2}}^{2} + C \| \rho \|_{L^{2}} \| \phi\|_{H^1} ( \| \mu_{2} \|_{L^{2}} +1)
\nonumber \\
& \le Ct^{\frac{1}{2}} \left( \| \mu_{x}\|_{L^{2}}^{\frac{1}{2}} + 1 \right) + Ct + Ct \left( \| \mu_{2}\|_{L^{2}} + 1 \right)
\nonumber \\
& \le Ct^{\frac{1}{4}} \left( t \| \mu_{x}\|_{L^{2}}^{2} \right)^{\frac{1}{4}} + Ct^{\frac{1}{2}} \left( t \| \mu_{x}\|_{L^{2}}^{2} \right)^{\frac{1}{2}} +  Ct^{\frac{1}{2}} \le Ct^{\frac{1}{4}}.
\end{align}

{\bf Step 3. Energy inequalities of $\phi$. }

Testing \eqref{nsch-w}$_{3}$ and \eqref{nsch-w}$_{4}$ by $-\mu$ and $\phi_{t}$ respectively and integrating the result over $I$ by parts, summing up the obtained results, noting
$$
\frac{1}{2} \frac{\rm d}{{\rm d}t}  \| \sqrt{\rho_{1}} \phi\|_{L^{2}}^{2} =   \frac{1}{2} \int_{I} \rho_{1t} \phi^{2} {\rm d}x + \int_{I} \rho_{1} \phi \phi_{t} {\rm d}x,
$$
we have
\begin{align}\label{wy-dt-3}
 & \frac{1}{2} \frac{\rm d}{{\rm d}t}  \left( \| \sqrt{\rho_{1}} \phi\|_{L^{2}}^{2} + \| \phi_{x}\|_{L^{2}}^{2}  \right) + \| \mu_{x}\|_{L^{2}}^{2}
 \nonumber \\
 = & \frac{1}{2} \int_{I} \rho_{1t} \phi^{2} {\rm d}x - \int_{I}  (\phi_{1}^{2}+ \phi_{1} \phi_{2} + \phi_{2}^{2}) \phi \rho_{1} \phi_{t} {\rm d}x - \int_{I} \rho_{1}u_{1} \phi_{x} \mu {\rm d}x - \int_{I} \rho_{1} u \phi_{2x}  \mu {\rm d}x
\nonumber \\
&   - \int_{I} \rho \phi_{2t} \mu {\rm d}x - \int_{I} \rho u_{2} \phi_{2x} \mu {\rm d}x   +  2\int_{I} \rho_{1} \phi \phi_{t} {\rm d}x + \int_{I} \rho [\mu_{2} - (\phi_{2}^{3}-\phi_{2}) ] \phi_{t} {\rm d}x
\nonumber \\
 =:& \sum\limits_{i=1}^{8} K_{i}.
\end{align}
We now estimate each term on the right hand side of \eqref{wy-dt-3} as follows. For $K_{1}$, we have
\begin{align*}
K_{1} & \le C \| \rho_{1t} \|_{L^{2}}  \| \phi \|_{L^{\infty}}^{2}  \le C \left( \| \sqrt{\rho_{1}} \phi \|_{L^{2}}^{2} +    \| \phi_{x} \|_{L^{2}}^{2} \right).
\end{align*}
Using \eqref{nsch-w}$_{3}$ and integration by parts, one has
\begin{align*}
K_{2} + K_{7} & = - \int_{I}  (\phi_{1}^{2}+ \phi_{1} \phi_{2} + \phi_{2}^{2}-2) \phi \left( - \rho_{1} u_{1} \phi_{x} + \mu_{xx} - \rho \phi_{2t} - \rho u_{2} \phi_{2x} - \rho_{1} u \phi_{2x}  \right)  {\rm d}x
\\
& = \int_{I}  (\phi_{1}^{2}+ \phi_{1} \phi_{2} + \phi_{2}^{2}-2) \phi \left( \rho_{1} u_{1} \phi_{x} + \rho \phi_{2t} + \rho u_{2} \phi_{2x} + \rho_{1} u \phi_{2x}  \right)  {\rm d}x
\\
& \quad +  \int_{I}  \partial_{x}(\phi_{1}^{2}+ \phi_{1} \phi_{2} + \phi_{2}^{2})  \phi \mu_{x} {\rm d}x +  \int_{I}  (\phi_{1}^{2}+ \phi_{1} \phi_{2} + \phi_{2}^{2}-2)  \phi_{x} \mu_{x} {\rm d}x
\\
& \le C\| \phi_{1}^{2}+ \phi_{1} \phi_{2} + \phi_{2}^{2} -2 \|_{L^{\infty}} \left( \| \sqrt{\rho_{1}} \phi \|_{L^{2}}\| \phi_{x} \|_{L^{2}} \| u_{1} \|_{L^{\infty}}  + \| \phi \|_{L^{\infty}} \| \rho \|_{L^{2}} \| \phi_{2t} \|_{L^{2}}  \right)
\\
& \quad + C\| \phi_{1}^{2}+ \phi_{1} \phi_{2} + \phi_{2}^{2} -2\|_{L^{\infty}}  \| \phi \|_{L^{\infty}} \| \rho \|_{L^{2}} \| u_{2} \|_{L^{\infty}}  \| \phi_{2x} \|_{L^{2}}
\\
& \quad + C\| \phi_{1}^{2}+ \phi_{1} \phi_{2} + \phi_{2}^{2} -2\|_{L^{\infty}}\| \sqrt{\rho_{1}} \phi \|_{L^{2}} \| \sqrt{\rho_{1}} u \|_{L^{2}} \| \phi_{2x} \|_{L^{\infty}}
\\
& \quad +  \| \partial_{x} (\phi_{1}^{2}+ \phi_{1} \phi_{2} + \phi_{2}^{2} ) \|_{L^{2}} \| \phi \|_{L^{\infty}} \| \mu_{x} \|_{L^{2}}
\\
& \quad +  \| \phi_{1}^{2}+ \phi_{1} \phi_{2} + \phi_{2}^{2}-2\|_{L^{\infty}} \| \phi_{x} \|_{L^{2}} \| \mu_{x} \|_{L^{2}}
\\
& \le C  \| \sqrt{\rho_{1}} \phi \|_{L^{2}}\| \phi_{x} \|_{L^{2}}  + C(\| \sqrt{\rho_{1}} \phi \|_{L^{2}} +    \| \phi_{x} \|_{L^{2}}) (t^{-1}\| \rho \|_{L^{2}} ) (t \| \phi_{2t} \|_{L^{2}} )
\\
& \quad + C (\| \sqrt{\rho_{1}} \phi \|_{L^{2}} +    \| \phi_{x} \|_{L^{2}}) \| \rho \|_{L^{2}} +  C \| \sqrt{\rho_{1}} \phi \|_{L^{2}} \| \sqrt{\rho_{1}} u \|_{L^{2}}
\\
& \quad + C  (\| \sqrt{\rho_{1}} \phi \|_{L^{2}} +    \| \phi_{x} \|_{L^{2}}) \| \mu_{x} \|_{L^{2}} + C \| \phi_{x} \|_{L^{2}} \| \mu_{x} \|_{L^{2}}
\\
& \le \frac{1}{2}\| \mu_{x} \|_{L^{2}}^{2} + C\frac{ \| \rho \|_{L^{2}}^{2} }{t} +  C \left(1 + t\| \phi_{2t} \|_{L^{2}}^{2}  \right)  \left( \| \sqrt{\rho_{1}} \phi \|_{L^{2}}^{2} +    \| \phi_{x} \|_{L^{2}}^{2} + \| \sqrt{\rho_{1}}u\|_{L^{2}}^{2} \right).
\end{align*}
For $K_{i} (i=3, \cdots, 6)$, direct calculation shows that
\begin{align*}
K_{3} & \le C \| \sqrt{\rho_{1}} \mu \|_{L^{2}} \| u_{1} \|_{L^{\infty}} \| \phi_{x} \|_{L^{2}} \le \varepsilon_{2}\| \sqrt{\rho_{1}} \mu \|_{L^{2}}^{2} + C\| \phi_{x} \|_{L^{2}}^{2},
\\[1em]
K_{4} & \le \| \sqrt{\rho_{1}} \mu \|_{L^{2}} \| \sqrt{\rho_{1}}u \|_{L^{2}} \| \phi_{2x} \|_{L^{\infty}} \le \varepsilon_{2}\| \sqrt{\rho_{1}} \mu \|_{L^{2}}^{2} + C \| \phi_{2x} \|_{H^{1}}^{2}  \| \sqrt{\rho_{1}}u \|_{L^{2}}^{2},
\\[1em]
K_{5} & \le C \| \rho \|_{L^{2}} \| \mu \|_{L^{\infty}} \| \phi_{2t} \|_{L^{2}} \le \varepsilon_{2} \left( \| \sqrt{\rho_{1}} \mu \|_{L^{2}}^{2}  + \| \mu_{x} \|_{L^{2}}^{2} \right) + C (t \| \phi_{2t} \|_{L^{2}}^{2}) \frac{\| \rho \|_{L^{2}}^{2}} {t},
\\[1em]
K_{6} & \le C \| \rho \|_{L^{2}} \| \mu \|_{L^{\infty}} \| u_{2} \|_{L^{\infty}}  \| \phi_{2x} \|_{L^{2}} \le \varepsilon_{2} \left( \| \sqrt{\rho_{1}} \mu \|_{L^{2}}^{2}  + \| \mu_{x} \|_{L^{2}}^{2} \right) + C\| \rho \|_{L^{2}}^{2}.
\end{align*}
Using \eqref{nsch-w}$_{1}$ and integration by parts, we obtain
\begin{align*}
K_{8} & = \int_{I} \rho [\mu_{2} - (\phi_{2}^{3}-\phi_{2}) ] \phi_{t} {\rm d}x
\\
&= \frac{\rm d}{{\rm d} t} \int_{I} \rho [\mu_{2} - (\phi_{2}^{3}-\phi_{2}) ] \phi {\rm d}x - \int_{I} \rho [\mu_{2t} - (3\phi_{2}^{2}-1)\phi_{2t} ] \phi {\rm d}x
\\
&\quad - \int_{I} \rho_{t} [\mu_{2} - (\phi_{2}^{3}-\phi_{2}) ] \phi {\rm d}x
\\
& = \frac{\rm d}{{\rm d} t} \int_{I} \rho [\mu_{2} - (\phi_{2}^{3}-\phi_{2}) ] \phi {\rm d}x   - \int_{I} \rho [\mu_{2t} - (3\phi_{2}^{2}-1)\phi_{2t} ] \phi {\rm d}x
\\
&\quad + \int_{I} [(\rho u_{1})_{x}  +   (\rho_{2} u)_{x}]  [\mu_{2} - (\phi_{2}^{3}-\phi_{2}) ] \phi {\rm d}x
\\
& = \frac{\rm d}{{\rm d} t} \int_{I} \rho [\mu_{2} - (\phi_{2}^{3}-\phi_{2}) ] \phi {\rm d}x  - \int_{I} \rho [\mu_{2t} - (3\phi_{2}^{2}-1)\phi_{2t} ] \phi {\rm d}x
\\
&\quad - \int_{I} [(\rho u_{1})  +   (\rho_{2} u) ]  [\mu_{2} - (\phi_{2}^{3}-\phi_{2}) ] \phi_{x} {\rm d}x
\\
& \quad - \int_{I} [(\rho u_{1})  +   (\rho_{2} u) ]  [\mu_{2x} - (3\phi_{2}^{2}-1) \phi_{2x} ] \phi {\rm d}x
\\
& \le \frac{\rm d}{{\rm d} t} \int_{I} \rho [\mu_{2} - (\phi_{2}^{3}-\phi_{2}) ] \phi {\rm d}x + \| \rho \|_{L^{2}} \| \phi \|_{L^{\infty}} \left( \| \mu_{2t} \|_{L^{2}} + \| 3\phi^{2}_{2} - 1 \|_{L^{\infty}}  \| \phi_{2t} \|_{L^{2}} \right)
\\
& \quad + \| \phi_{x} \|_{L^{2}} \left( \| \rho \|_{L^{2}}  \| u_{1} \|_{L^{\infty}} + \| \rho_{2} \|_{L^{\infty}}  \| u \|_{L^{2}} \right) \left( \| \mu_{2}\|_{L^{\infty}} +  \| \phi^{3}_{2} - \phi_{2} \|_{L^{\infty}}  \right)
\\
& \quad + \| \phi \|_{L^{\infty}} \left( \| \rho \|_{L^{2}}  \| u_{1} \|_{L^{\infty}} + \| \rho_{2} \|_{L^{\infty}}  \| u \|_{L^{2}} \right) \left( \| \mu_{2x}\|_{L^{2}} +  \| 3\phi^{2}_{2} - 1 \|_{L^{\infty}}  \| \phi_{2x} \|_{L^{2}} \right)
\\
& \le \frac{\rm d}{{\rm d} t} \int_{I} \rho [\mu_{2} - (\phi_{2}^{3}-\phi_{2}) ] \phi {\rm d}x  +  C \frac{\| \rho \|_{L^{2}} } {t} \left( \| \sqrt{\rho_{1}} \phi \|_{L^{2}} +    \| \phi_{x} \|_{L^{2}} \right)  \left( t\| \mu_{2t} \|_{L^{2}} +  t\| \phi_{2t} \|_{L^{2}} \right)
\\
& \quad + C \left( \| \sqrt{\rho_{1}} \phi \|_{L^{2}} + \| \phi_{x} \|_{L^{2}}  \right) \left( \| \rho \|_{L^{2}} +  \| u_{x} \|_{L^{2}} \right) \left( \| \mu_{2}\|_{H^{1}} +  1 \right)
\\
& \le \frac{\rm d}{{\rm d} t} \int_{I} \rho [\mu_{2} - (\phi_{2}^{3}-\phi_{2}) ] \phi {\rm d}x  +   \frac{1}{2} \frac{\| \rho \|_{L^{2}}^{2} } {t^{2}} + \frac{1}{4}  \| u_{x} \|_{L^{2}}^{2}
\\
& \quad + C \left( \| \sqrt{\rho_{1}} \phi \|_{L^{2}}^{2} +    \| \phi_{x} \|_{L^{2}}^{2} + \| \rho \|_{L^{2}}^{2} \right)  \left( t^{2}\| \mu_{2t} \|_{L^{2}}^{2} +  t^{2}\| \phi_{2t} \|_{L^{2}}^{2} + \| \mu_{2}\|_{H^{1}}^{2} + 1 \right).
\end{align*}
Substituting $K_{1}$--$K_{8}$ into \eqref{wy-dt-3}, we arrive at
\begin{align}\label{wy-phi}
 & \frac{\rm d}{{\rm d}t}  \left( \| \sqrt{\rho_{1}} \phi\|_{L^{2}}^{2} + \| \phi_{x}\|_{L^{2}}^{2}  \right) + \| \mu_{x}\|_{L^{2}}^{2}
 \nonumber \\
 \le& 2\frac{\rm d}{{\rm d} t} \int_{I} \rho [\mu_{2} - (\phi_{2}^{3}-\phi_{2}) ] \phi {\rm d}x  +   \frac{1}{2}  \frac{\| \rho \|_{L^{2}}^{2} } {t^{2}} + \frac{1}{4}  \| u_{x} \|_{L^{2}}^{2} + C\varepsilon_{2} \left( \| \sqrt{\rho_{1}} \mu \|_{L^{2}}^{2}  + \| \mu_{x} \|_{L^{2}}^{2} \right)
\nonumber \\
&  + C \mathcal{B}(t) \left( \frac{ \| \rho \|_{L^{2}}^{2} }{t} +  \| \sqrt{\rho_{1}} \phi \|_{L^{2}}^{2} +    \| \phi_{x} \|_{L^{2}}^{2} + \| \sqrt{\rho_{1}} u\|_{L^{2}}^{2} \right),
\end{align}
where
\begin{align*}
\mathcal{B}(t) := 1 + \| \mu_{2}\|_{H^{1}}^{2} + \| \phi_{2x} \|_{H^{1}}^{2} + t^{2}\| \phi_{2t} \|_{L^{2}}^{2} +  t^{2}\| \mu_{2t} \|_{L^{2}}^{2}.
\end{align*}
We next estimate $\| \sqrt{\rho_{1}} \mu \|_{L^{2}}$. Testing \eqref{nsch-w}$_{4}$ by $\mu$,  integrating over $I$ by parts, one gets
\begin{align}\label{wy-mu}
\| \sqrt{\rho_{1}} \mu \|_{L^{2}}^{2} & = \int_{I} \phi_{x} \mu_{x} {\rm d}x + \int_{I} \rho_{1}\mu (\phi_{1}^{2}+ \phi_{1} \phi_{2} + \phi_{2}^{2} -1) \phi  {\rm d}x
\nonumber \\
& \quad +  \int_{I} \rho [(\phi_{2}^{3}-\phi_{2}) -  \mu_{2} ] \mu  {\rm d}x
\nonumber \\
& \le \frac{1}{4} \| \sqrt{\rho_{1}} \mu \|_{L^{2}}^{2} + \| \phi_{x} \|_{L^{2}}^{2} + \| \mu_{x} \|_{L^{2}}^{2} + C\| \sqrt{\rho_{1}} \phi \|_{L^{2}}^{2}
\nonumber \\
& \quad + C \frac{1}{\sqrt{t}} \| \rho\|_{L^{2}} \sqrt{t} \left( 1+  \| \mu_{2}\|_{L^{2}} \right) \| \mu \|_{L^{\infty}}
\nonumber \\
& \le \frac{1}{4} \| \sqrt{\rho_{1}} \mu \|_{L^{2}}^{2} + \| \phi_{x} \|_{L^{2}}^{2} + \| \mu_{x} \|_{L^{2}}^{2} + C\| \sqrt{\rho_{1}} \phi \|_{L^{2}}^{2}
\nonumber \\
& \quad + C \frac{1}{\sqrt{t}} \| \rho\|_{L^{2}} \left(  \|  \sqrt{\rho_{1}} \mu \|_{L^{2}} + \|   \mu_{x} \|_{L^{2}} \right)
\nonumber \\
& \le \frac{1}{2} \| \sqrt{\rho_{1}} \mu \|_{L^{2}}^{2} + \| \phi_{x} \|_{L^{2}}^{2} + C \| \mu_{x} \|_{L^{2}}^{2} + C\| \sqrt{\rho_{1}} \phi \|_{L^{2}}^{2} + C\frac{\| \rho\|_{L^{2}}^{2}}{t}.
\end{align}
Inserting \eqref{wy-mu} into \eqref{wy-phi}, integrating over $(0,t)$, choosing $\varepsilon_{2}$ sufficiently small such that $C\varepsilon_{2}<1/2$, and using \eqref{phi000} and \eqref{zz-phix},  there holds
\begin{align}\label{wy-phi-tt}
& \| \sqrt{\rho_{1}} \phi\|_{L^{2}}^{2} + \| \phi_{x}\|_{L^{2}}^{2}  + \int_{0}^{t} \| \mu_{x}\|_{L^{2}}^{2} {\rm d}s
 \nonumber \\
 \le& C_{3}\frac{ \| \rho\|_{L^{2}}^{2} }{t}  +   \frac{1}{2}  \int_{0}^{t} \frac {\| \rho \|_{L^{2}}^{2} } {s^{2}} {\rm d}s  + \frac{1}{4}  \int_{0}^{t}  \| u_{x} \|_{L^{2}}^{2} {\rm d}s
\nonumber \\
&  + C \int_{0}^{t} \mathcal{B}(s) \left( \frac{ \| \rho \|_{L^{2}}^{2} }{s} +  \| \sqrt{\rho_{1}} \phi \|_{L^{2}}^{2} +    \| \phi_{x} \|_{L^{2}}^{2} + \| \sqrt{\rho_{1}} u\|_{L^{2}}^{2} \right) {\rm d}s,
\end{align}
where we have used
\begin{align*}
& 2 \int_{I} \rho [\mu_{2} - (\phi_{2}^{3}-\phi_{2}) ] \phi {\rm d}x
\\
 \le& 2 \| \rho\|_{L^{2}} \| \phi \|_{L^{\infty}} (\| \mu_{2}\|_{L^{2}} + \| \phi_{2}^{3} -\phi_{2}\|_{L^{2}})
\\
 \le& C \frac{ \| \rho\|_{L^{2}} }{\sqrt{t}}\left( \| \sqrt{\rho_{1}} \phi\|_{L^{2}} + \| \phi_{x}\|_{L^{2}} \right)  \sqrt{t} \| \mu_{2}\|_{L^{2}} + C\| \rho\|_{L^{2}} \left( \| \sqrt{\rho_{1}} \phi\|_{L^{2}} + \| \phi_{x}\|_{L^{2}} \right)
\\
 \le& \frac{1}{2}\left( \| \sqrt{\rho_{1}} \phi\|_{L^{2}}^{2} + \| \phi_{x}\|_{L^{2}}^{2} \right) + C_{3}\frac{ \| \rho\|_{L^{2}}^{2} }{t}.
\end{align*}

{\bf Step 4. Singular t-weighted energy inequalities and uniqueness.}

Similar to \eqref{wy-00}, one has
\begin{align}\label{wy-000}
 \frac{1}{2} \frac{\rm d}{{\rm d}t} \| \rho \|_{L^{2}}^{2}
 & \le  \frac{1}{2} \| u_{1x} \|_{L^{\infty}} \| \rho \|_{L^{2}}^{2} + \| u \|_{L^{\infty}} \| \rho_{2x} \|_{L^{2}}  \| \rho \|_{L^{2}}  +  \| \rho_{2} \|_{L^{\infty}} \| u_{x} \|_{L^{2}}  \| \rho \|_{L^{2}}
\nonumber \\
 & \le C  \| u_{1x} \|_{H^{1}} \| \rho \|_{L^{2}}^{2}  + C t \| u_{x} \|_{L^{2}}^{2} + \frac{1}{2} \frac{\| \rho \|_{L^{2}}^{2}}{t}.
\end{align}
Multiplying \eqref{wy-000} by $t^{-1}$ yields
\begin{align}\label{wy0}
\frac{\rm d}{{\rm d}t} \left( \frac{\| \rho \|_{L^{2}}^{2}} { t } \right) +  \frac{\| \rho \|_{L^{2}}^{2}} {t^{2}}  & \le C \left( \| u_{1x} \|_{H^{1}} + 1 \right) \frac{\| \rho \|_{L^{2}}^{2}} { t } + C_{1}  \| u_{x} \|_{L^{2}}^{2}.
\end{align}
Multiplying \eqref{wy-11} by $2(C_{1} + 1)$,  and adding the resultant with \eqref{wy0}, one obtains
\begin{align}\label{wy1}
& \frac{\rm d}{{\rm d}t} \left( \frac{\| \rho \|_{L^{2}}^{2}} { t } + (C_{1}+1)  \| \sqrt{\rho_{1}} u \|_{L^{2}}^{2}  \right) +  \frac{\| \rho \|_{L^{2}}^{2}} {t^{2}} + \| u_{x}\|_{L^{2}}^{2}
\nonumber \\
\le&  C \left( \| u_{1x} \|_{H^{1}} + t\| u_{2t} \|_{L^{2}}^{2} + \| \phi_{1x}\|_{H^{1}}^{2} + \| \phi_{2xx}\|_{L^{2}}^{2} + 1 \right)  \left( \frac{\| \rho \|_{L^{2}}^{2}} { t }  +  \| \sqrt{\rho_{1}} u \|_{L^{2}}^{2} + \| \phi_{x}\|_{L^{2}}^{2}\right).
\end{align}
Multiplying \eqref{wy1} by $C_{3} + 1$, integrating it over $(0,t)$, adding the resultant with \eqref{wy-phi-tt}, we finally arrive at
\begin{align}\label{wy-zong}
& \frac{\| \rho \|_{L^{2}}^{2}} { t } +   \| \sqrt{\rho_{1}} u \|_{L^{2}}^{2} + \| \sqrt{\rho_{1}} \phi\|_{L^{2}}^{2} + \| \phi_{x}\|_{L^{2}}^{2}  + \int_{0}^{t} \left( \frac{\| \rho \|_{L^{2}}^{2}} {s^{2}} + \| u_{x}\|_{L^{2}}^{2} + \| \mu_{x}\|_{L^{2}}^{2} \right) {\rm d}s
 \nonumber \\
 \le&   C \int_{0}^{t} \mathcal{F}(s) \left( \frac{ \| \rho \|_{L^{2}}^{2} }{s}  + \| \sqrt{\rho_{1}} u\|_{L^{2}}^{2}  +  \| \sqrt{\rho_{1}} \phi \|_{L^{2}}^{2} +    \| \phi_{x} \|_{L^{2}}^{2}\right) {\rm d}s,
\end{align}
where
\begin{align*}
\mathcal{F}(s) : = 1 + \| \mu_{2}\|_{H^{1}}^{2} + \| \phi_{2x} \|_{H^{1}}^{2} + t^{2}\| \phi_{2t} \|_{L^{2}}^{2} +  t^{2}\| \mu_{2t} \|_{L^{2}}^{2}+ \| u_{1x} \|_{H^{1}} + t\| u_{2t} \|_{L^{2}}^{2} + \| \phi_{1x}\|_{H^{1}}^{2},
\end{align*}
and $\mathcal{F}(s) \in L^{1}(0,t)$. By using \eqref{fzz-rho}, \eqref{u000}, \eqref{phi000} and \eqref{zz-phix}, we conclude that
\begin{align*}
\lim\limits_{t \to 0} \left( \frac{\| \rho \|_{L^{2}}^{2}} { t } +   \| \sqrt{\rho_{1}} u \|_{L^{2}}^{2} + \| \sqrt{\rho_{1}} \phi\|_{L^{2}}^{2} + \| \phi_{x}\|_{L^{2}}^{2} \right) (t) = 0.
\end{align*}
Thus, one deduces from Gronwall's inequality that
\begin{align*}
& \left( \frac{\| \rho \|_{L^{2}}^{2}} { t } +   \| \sqrt{\rho_{1}} u \|_{L^{2}}^{2} + \| \sqrt{\rho_{1}} \phi\|_{L^{2}}^{2} + \| \phi_{x}\|_{L^{2}}^{2} \right)(t) + \int_{0}^{t} \left( \frac{\| \rho \|_{L^{2}}^{2}} {s^{2}} + \| u_{x}\|_{L^{2}}^{2} + \| \mu_{x}\|_{L^{2}}^{2} \right) {\rm d}s =0.
\end{align*}
which implies $\rho = u = \phi = 0$. As a consequence, the uniqueness of solutions is proved.
\end{proof}

\smallskip
{\bf Acknowledgment.}
Ding's work is supported by the Key Project of the National Natural Science Foundation of China (No.12131010), the National Natural Science Foundation of China (No.12271032) and Guangdong Basic and Applied Basic Research Foundation (No. 2026A1515010778). Li's work is supported by the National Natural Science Foundation of China (No.12371205) and the Natural Science Foundation of Guangdong Province (No.2025A151501202 6, 2025A1515040001). Yan's work is supported by the National Natural Science Foundation of China (No.12526617).

\end{document}